\documentclass[11pt,reqno]{amsart}
\usepackage{amssymb,amsmath}
\usepackage{amsthm}
\usepackage{color,graphicx}
\usepackage{hyperref}

\newcommand{\R}{\mathbb R}
\newcommand{\N}{{\mathbb N}}

\newcommand{\C}{\mathbb C}
\newcommand{\wtw}{\widetilde{W}}
\newcommand{\wtphi}{\widetilde{\phi}}
\newcommand{\wtpsi}{\widetilde{\psi}}

\newcommand{\p}{\partial}

\newcommand{\na}{\nabla}

\newcommand{\Hmeio}{\dot{H}^{1/2}}

\newcommand{\ff}{\varphi}
\newcommand{\ub}{\overline{u}}
\newcommand{\vb}{\overline{v}}

\numberwithin{equation}{section}

\newtheorem{theorem}{Theorem}[section]
\newtheorem{proposition}[theorem]{Proposition}
\newtheorem{remark}[theorem]{Remark}
\newtheorem{lemma}[theorem]{Lemma}
\newtheorem{corollary}[theorem]{Corollary}
\newtheorem{definition}[theorem]{Definition}

\begin{document}
\vglue-1cm \hskip1cm
\title[]{Scattering for a 3D
Coupled Nonlinear Schr\"odinger system}

\author[L. G. Farah]{Luiz G. Farah}
\address{ICEx, Universidade Federal de Minas Gerais, Av. Ant\^onio Carlos, 6627, Caixa Postal 702, 30123-970,
Belo Horizonte-MG, Brazil}
\email{lgfarah@gmail.com}

\author[A. Pastor]{Ademir Pastor}
\address{IMECC-UNICAMP, Rua S\'ergio Buarque de Holanda, 651, 13083-859, Cam\-pi\-nas-SP, Bra\-zil}
\email{apastor@ime.unicamp.br}




\begin{abstract}
 We consider the three-dimensional cubic nonlinear
Schr\"odinger system
 \begin{equation*}
\begin{cases}
i\partial_tu+\Delta u+(|u|^2+\beta |v|^2)u=0,\\
i\partial_tv+\Delta v+(|v|^2+\beta |u|^2)v=0.
\end{cases}
\end{equation*}
Let $(P,Q)$ be any ground state solution of the above Schr\"odinger system. We show that for any initial data $(u_0,v_0)$ in
$H^1(\mathbb{R}^3)\times H^1(\mathbb{R}^3)$ satisfying $M(u_0,v_0)A(u_0,v_0)<M(P,Q)A(P,Q)$ and $M(u_0,v_0)E(u_0,v_0)<M(P,Q)E(P,Q)$, where $M(u,v)$ and $E(u,v)$ are the mass and energy (invariant quantities) associated to the system, the corresponding solution is global in $H^1(\mathbb{R}^3)\times H^1(\mathbb{R}^3)$ and scatters. Our approach  is in the same spirit of Duyckaerts-Holmer-Roudenko \cite{dhr}, where the authors considered the 3D cubic nonlinear Schr\"odinger equation.

\end{abstract}

\maketitle

\section{Introduction}\label{introduction}

This work is concerned with the following Cauchy problem associated with a three-dimensional cubic nonlinear Schr\"odinger (NLS) system:
\begin{equation}\label{3nls}
\begin{cases}
i\partial_tu+\Delta u+(|u|^2+\beta |v|^2)u=0,\\
i\partial_tv+\Delta v+(|v|^2+\beta |u|^2)v=0,\\
u(x,0)=u_0(x), \quad v(x,0)=v_0(x),
\end{cases}
\end{equation}
where $(x,t)\in \R^3\times \R$, $u=u(x,t)$ and $v=v(x,t)$ are complex-valued
functions, and $\beta>0$ is a real coupling parameter. The system appears, for instance, in the propagation of laser beams in birefringent Kerr medium in nonlinear optics (see \cite{ag}). Besides, a large quantity of nonlinear physical phenomena can be modelled using the system of equations in \eqref{3nls}, which we refrain from list them here (see e.g., \cite{ak}, \cite{ag}, and references therein).

Our main interest here concerns the behavior  of the solutions of \eqref{3nls}. Let us start by observing that the system conserves the quantities
\begin{equation}\label{masscon}
M(u,v)=\int_{\R^3}(|u|^2+|v|^2)dx,
\end{equation}
\begin{equation}\label{mocon}
F(u,v)=Im\int_{\R^3}(\overline{u}\nabla u+\overline{v}\nabla v)dx,
\end{equation}
and
\begin{equation}\label{econ}
E(u,v)=\frac{1}{2}\int_{\R^3}(|\na u|^2+|\na
v|^2)dx-\frac{1}{4}\int_{\R^3}(|u|^{4}+2\beta|uv|^{2}+|v|^{4})dx.
\end{equation}
This means if $(u(t),v(t))$ is a sufficiently regular solution of \eqref{3nls}, in some lifespan interval $\mathcal{I}$, then $M(u(t),v(t))=M(u_0,v_0)$, $F(u(t),v(t))=F(u_0,v_0)$, and $E(u(t),v(t))=E(u_0,v_0)$, for all $t\in \mathcal{I}$.
The invariants \eqref{masscon}-\eqref{econ} are the so called mass, momentum, and energy. In order to simplify notation, we also define
\begin{equation}\label{acon}
A(u,v)=\int_{\R^3}(|\na u|^2+|\na
v|^2)dx.
\end{equation}

The system \eqref{3nls} enjoys a scaling invariance symmetry, which says that if $(u,v)$ is a solution of \eqref{3nls}, then
\begin{equation}\label{scalsym}
u_\lambda(x,t)=\lambda u(\lambda x,\lambda^2t), \qquad v_\lambda(x,t)=\lambda
v(\lambda x,\lambda^2t),
\end{equation}
is also a solution with initial data $\big(\lambda u_0(\lambda x), \lambda
v_0(\lambda x)\big)$, for any $\lambda>0$.  A simple computation reveals that
$$
\|u_\lambda(\cdot,0)\|_{\Hmeio}+\|v_\lambda(\cdot,0)\|_{\Hmeio}=\|u_0\|_{\Hmeio}+\|v_0\|_{\Hmeio},
$$
where $\Hmeio=\Hmeio(\R^3)$ is the homogeneous $L^2$-based Sobolev space of order $1/2$  (see notations below for the definition).
Thus, $\Hmeio(\R^3)\times\Hmeio(\R^3)$ is the scale-invariant Sobolev space for system \eqref{3nls}. Such a space is,
therefore, called \textit{critical space}. 

It is easy to see that the quantities $M(u,v)E(u,v)$ and $M(u,v)A(u,v)$ are also invariant under the same scaling, indeed $M(u_\lambda,v_\lambda)E(u_\lambda,v_\lambda)= M(u,v)E(u,v)$ and also $M(u_\lambda,v_\lambda)A(u_\lambda,v_\lambda)=M(u,v)A(u,v)$. In particular, we have
\begin{equation}\label{scalL2}
M(u_\lambda,v_\lambda)=\frac{1}{\lambda}M(u,v).
\end{equation} 

Another invariance associated with \eqref{3nls} is the so called \textit{Galilean invariance}, that is, if $(u,v)$ is a solution  then, for any $\xi_0\in \R^3$,
\begin{equation}\label{GaIn}
(e^{i(x\cdot\xi_0-t|\xi|^2)}u(x-2t\xi_0, t), e^{i(x\cdot\xi_0-t|\xi|^2)}v(x-2t\xi_0, t))
\end{equation} 
is also a solution with initial data $(e^{ix\cdot\xi_0}u_0(x), e^{ix\cdot\xi_0}v_0(x))$.

The system \eqref{3nls} has some very special solutions. Indeed, if we look for  solutions of the form
$$
(u(x,t),v(x,t))=(e^{it}P(x), e^{it}Q(x))
$$
then $(P,Q)$ must solve the following 3D elliptic system:
\begin{equation}\label{ellip1}
\begin{cases}
-\Delta P+P-(|P|^{2}+\beta |Q|^{2})P=0,\\
-\Delta Q+Q-(|Q|^{2}+\beta |P|^{2})Q=0.
\end{cases}
\end{equation}

A nontrivial solution $(P,Q)\neq (0,0)$ which has the least energy level is commonly refereed to as a \textit{ground state solution}. The existence of positive, radially symmetric, ground state solutions for \eqref{ellip1} was studied, for instance, by Maia-Montefusco-Pellacci \cite{mmp}. Such a solution plays a crucial role in our further analysis (threshold for global existence and blow up of the solution). Throughout the paper $(P,Q)$ always will denote the ground state solution of \eqref{ellip1}. Although uniqueness of the ground state is not known for any $\beta$ (see \cite{wy}), this will not be an issue to our purpose (see also our Remark \ref{remgn} below).

A local well-posedness result for \eqref{3nls} can be established by combining the well-known Strichatz estimates (see Lemma \ref{strilem} below) with the contraction mapping principle (see \cite{caz}). In addition, such a local solution can be extended globally in time, if one impose some additional condition on the initial data. More precisely, in \cite{P2015} the second author obtained the following global well-posedness result (see \cite{hr} and \cite{hr1} for the similar result concerning the NLS equation).

\begin{theorem}\label{globalPastor}
Let  $(u,v)\in C((-T_*,T^*);H^1(\R^3)\times H^1(\R^3))$ be the solution of \eqref{3nls}
with initial data $(u_0,v_0)\in H^1(\R^3)\times H^1(\R^3)$, where $I:=(-T_*,T^*)$ is the
maximal time interval of existence. Assume that
\begin{equation}\label{EM}
M(u_0,v_0)E(u_0,v_0)<M(P,Q)E(P,Q).
\end{equation}
If
\begin{equation}\label{EM1}
M(u_0,v_0)A(u_0,v_0)<M(P,Q)A(P,Q),
\end{equation}
then
\begin{equation}\label{EM2}
M(u(t),v(t))A(u(t),v(t))<M(P,Q)A(P,Q)
\end{equation}
and the solution exists globally in time, that is, $I=(-\infty, \infty)$.
\end{theorem}

\begin{remark}\label{Rema}
Note that under the assumptions of Theorem \ref{globalPastor}, we also have 
$$
(u,v)\in L^{\infty}(\R;H^1(\R^3)\times H^1(\R^3)),
$$
in view of \eqref{EM2} and the fact that the mass is conserved.
\end{remark}

\begin{remark}\label{BlowCond}
It was also proved in \cite[Theorem 1.2]{P2015} that if $u_0$ and $v_0$ are radial functions satisfying $M(u_0,v_0)E(u_0,v_0)<M(P,Q)E(P,Q)$ and also $M(u_0,v_0)A(u_0,v_0)>M(P,Q)A(P,Q)$ then the solution blows up in finite time.
\end{remark}

Our goal in the present work is to show that under the assumptions \eqref{EM} and \eqref{EM1} the global solution $(u(t),v(t))$ given in Theorem \ref{globalPastor}  scatters in the  sense of Definition \ref{scatdef} below. In what follows, we denote by $e^{it\Delta}$  the unitary group associated with the linear Schr\"odinger equation $iu_t+\Delta u=0$.

\begin{definition}\label{scatdef}
We say that a global solution $(u(t),v(t))$ scatters forward in time if there exist $\phi^+,\psi^+\in H^1(\R^3)$ such that 
\begin{equation}\label{defscat}
\lim_{t\to+\infty}\|(u(t),v(t))-(e^{it\Delta}\phi^+,e^{it\Delta}\psi^+)\|_{H^1\times H^1}=0.
\end{equation}
Also, we say that  $(u(t),v(t))$ scatters backward in time if there exist $\phi^-,\psi^-\in H^1(\R^3)$ such that 
\begin{equation}\label{defscat-}
\lim_{t\to-\infty}\|(u(t),v(t))-(e^{it\Delta}\phi^-,e^{it\Delta}\psi^-)\|_{H^1\times H^1}=0.
\end{equation}
\end{definition}

The main result of our study is the following one.
\begin{theorem}\label{globalscatters}
Assume that  $(u_0,v_0)\in \mathcal{K}$, where
\begin{equation*}
\mathcal{K}:=\left\{(u_0, v_0)\in H^1(\R^3)\times  H^1(\R^3):\;  \textrm{relations} \; \eqref{EM} \; \textrm{and} \; \eqref{EM1} \; \textrm{hold}\;\right\}.
\end{equation*}
 Then the corresponding solution $(u(t),v(t))$ of system \eqref{3nls} exists globally in time and scatters both forward and backward in time.
\end{theorem}

Our strategy to prove Theorem \ref{globalscatters} is to adapt the ideas of Duyckaerts-Holmer-Roudenko \cite{dhr} (see also \cite{hr}), where the authors considered the 3D cubic nonlinear Schr\"odinger equation. It worth mentioning that the core of the approach in \cite{dhr} was prior introduced by Kenig-Merle \cite{km} in the study of the energy-critical Schr\"odinger equation. To the best of our knowledge, our paper is the first one to deal with these techniques in order to show scattering for coupled systems of nonlinear Schr\"odinger equations under the conditions given in Theorem \ref{globalPastor} (see \eqref{EM}-\eqref{EM1} and also Remark \ref{BlowCond}). For the sake of simplicity, we will only prove scattering forward in time; the case of backward in time follows the same plan.

It should be noted that, by using a different approach, scattering for \eqref{3nls} has appeared in Xu \cite{Xu12}. However, as we describe below our result extend the one in \cite{Xu12}. Indeed, consider the functionals
\[
\begin{split}
J(u,v)&=\frac{1}{2}\int_{\R^3}(|\na u|^2+|\na
v|^2+|u|^2+|v|^2)dx-\frac{1}{4}\int_{\R^3}(|u|^{4}+2\beta|uv|^{2}+|v|^{4})dx\\
& =E(u,v)+\frac{1}{2}M(u,v)
\end{split}
\]
and
\[
\begin{split}
K(u,v)&=2\int_{\R^3}(|\na u|^2+|\na
v|^2)dx-\frac{3}{2}\int_{\R^3}(|u|^{4}+2\beta|uv|^{2}+|v|^{4})dx\\
&=2A(u,v)-\frac{3}{2}\int_{\R^3}(|u|^{4}+2\beta|uv|^{2}+|v|^{4})dx.
\end{split}
\]
Also, let
$$
J_0:=\inf\Big\{J(u,v): \,(u,v)\in H^1(\R^3)\times  H^1(\R^3)\setminus\{(0,0)\}, K(u,v)=0\Big\}\equiv J(P,Q)>0
$$
and define the sets
$$
K^+=\{(u,v)\in H^1(\R^3)\times  H^1(\R^3): J(u,v)<J_0, K(u,v)\geq0\}
$$
$$
K^-=\{(u,v)\in H^1(\R^3)\times  H^1(\R^3: J(u,v)<J_0, K(u,v)<0\}.
$$
The author in \cite{Xu12} then showed if $(u_0,v_0)\in K^+$, then the corresponding solution is global and scatters. Also, if $(u_0,v_0)\in K^-$ is radial or $(xu_0,xv_0)\in L^2(\R^3)\times L^2(\R^3)$, then the solution blows up in finite time. In particular, his result establishes scattering only for initial data in $K^+$. However, as we show in our Appendix,
\begin{equation}\label{subxu}
K^+\varsubsetneq \mathcal{K},
\end{equation}
by proving that the set where we have scattering is larger than the one obtained in \cite{Xu12}.

At least from the mathematical point of view, system \eqref{3nls} can be generalized to 
\begin{equation}\label{3nlsgen}
\begin{cases}
i\partial_tu+\Delta u+\mu(|u|^{2p}+\beta |u|^{p-1}|v|^{p+1})u=0,\\
i\partial_tv+\Delta v+\mu(|v|^{2p}+\beta |u|^{p+1}|v|^{p-1})v=0,\\
u(x,0)=u_0(x), \quad v(x,0)=v_0(x),
\end{cases}
\end{equation}
where $(x,t)\in\R^n\times \R$, $u=u(x,t)$, $v=v(x,t)$, $p>1$, and $\mu=1$ (\textit{focusing case}) or $\mu=-1$ (\textit{defocusing case}). A global well-posedness for \eqref{3nlsgen} in the focusing case, in the same spirit of Theorem \ref{globalPastor}, was established in \cite{lwl}. The sharp threshold for the global existence now depends on the ground state solutions of the elliptic system associated with \eqref{3nlsgen}. The interested reader will also find some related work in \cite{chb}, \cite{ll}, \cite{maz}, \cite{ntds}, \cite{pvp}, \cite{so}, and references therein. We believe that our results in Theorem \ref{globalscatters} can also be generalized to this case.  A possible approach to obtain the results could be to extend the results in \cite{fxc} and \cite{gue}, where the authors dealt with the multi-dimensional Schr\"odinger equation with a power-law nonlinearity. This will be issue for further investigation.

In the defocusing case, scattering theory for \eqref{3nlsgen} has already been established in  \cite{ct}. The approach is completely different from ours; most of the arguments are based on Morawetz-type identities and inequalities and the corresponding Morawetz estimates. In that case, all solutions with initial data in $H^1(\R^3)\times  H^1(\R^3)$ are global and scatters forward and backward in time.

The paper is organized as follows. In Section \ref{sec2} we introduce some notation and recall the Strichartz estimates for the Schr\"odinger equation. In Section \ref{sec3}, we first recall the sharp Gagliardo-Nirenberg inequality; which is fundamental to obtain our results. In the sequel we recall and prove some results concerning the Cauchy problem \eqref{3nls}. In particular a small data global result and a long time perturbation theory are established. Also, we recall a sufficient condition for proving the scattering. In Section \ref{sec4}, we establish the profile and energy decomposition of bounded sequences in $H^1(\R^3)$. In particular, in the profile decomposition, we show that space and time shifts can be taken to be the same for two distinct sequences. The proof Theorem \ref{globalscatters} is essentially initiated in Section \ref{ConstCri}, where we reformulate the conclusion in an equivalent way and prove the existence of the so called critical solution. In Section \ref{sec6}, we complete the proof of Theorem \ref{globalscatters}, by showing a Liouville-type theorem and concluding that the critical solution constructed previously cannot exist.

\section{Notations and Preliminaries}\label{sec2}

Let us start this section by introducing the notation used throughout the paper. We use $c$ to denote various constants that may vary line by line. For any positive quantities $a$ and $b$, the notation $a \lesssim b$ means
that $a \leq cb$, with $c$ uniform with respect to the set where $a$ and $b$ vary. Given a complex number $z$, we use $Re(z)$ and $Im(z)$ to denote, respectively, the real and imaginary parts of $z$. Otherwise is stated, $\int f$ always mean integration of
the function $f$ over all $\R^3$.

We use $\|\cdot\|_{L^p_x}$ or $\|\cdot\|_{L^p}$ to denote the $L^p(\R^3)$ norm. If necessary, we use subscript to inform which variable we are concerned with. The mixed norms $L^q_tL^r_x$ of $f=f(x,t)$ is defined as
\begin{equation*}
\|f\|_{L^q_tL^r_x}= \left(\int_{-\infty}^{+\infty} \|f(\cdot,t)\|_{L^r_x}^q dt \right)^{1/q}
\end{equation*}
with the usual modifications when $q =\infty$ or $r=\infty$.  Similarly, we also define the norms in the spaces  $L^r_xL^q_t$. When $q=r$, for short we sometimes denote $\|f\|_{L^r_tL^r_x}$ by
$\|f\|_{L^r_{x,t}}$. For an interval $I\subset\R$, in a similar fashion we define $\|f\|_{L^q_IL^r_x}$.

The spatial Fourier transform of $f(x)$ is given by
\begin{equation*}
\hat{f}(\xi)=\int e^{-ix\cdot\xi}f(x)dx.
\end{equation*}


For any $s\in\R$, we shall define $D^s$ and $J^s$ to be, respectively, the Fourier
multiplier with symbol $|\xi|^s$ and $\langle \xi \rangle^s = (1+|\xi|)^s$.
In this case, the norm in the Sobolev spaces $H^s:=H^s(\R^3)$ and $\dot{H}^s:=\dot{H}^s(\R^3)$
are given, respectively, by
\begin{equation*}
\|f\|_{H^s}\equiv \|J^sf\|_{L^2_x}=\|\langle \xi
\rangle^s\hat{f}\|_{L^2_{\xi}}, \qquad \|f\|_{\dot{H}^s}\equiv
\|D^sf\|_{L^2_x}=\||\xi|^s\hat{f}\|_{L^2_{\xi}}.
\end{equation*}
The product space $H^s\times H^s$ is equipped with the norm $\|(f,g)\|^2_{H^s\times H^s}\equiv \|f\|^2_{H^s}+\|g\|^2_{H^s}$. In an analogous way we define the norm in $\dot{H}^s\times \dot{H}^s$. 

In order to  study the Cauchy problem associated with the NLS system \eqref{3nls}, let us
introduce some functional spaces. Given any $s\in\R$, we say that a pair $(q,r)$ is
$\dot{H}^s$ admissible if
$$
\frac{2}{q}+\frac{3}{r}=\frac{3}{2}-s.
$$
In what follows, we denote
$$
\|u\|_{S(L^2)}=\sup\{\|u\|_{L^q_tL^r_x};\; (q,r)\, {\rm is}\, L^2\; {\rm
admissible},\; 2\leq r\leq6,\; 2\leq q\leq\infty\}
$$
and
$$
\|u\|_{S(\Hmeio)}=\sup\{\|u\|_{L^q_tL^r_x};\;(q,r)\,{\rm is}\, \Hmeio\; {\rm
admissible}, \,3\leq r\leq6^-,\; 4^+\leq q\leq\infty\}.
$$
We also need to consider the dual norms:
$$
\|u\|_{S'(L^2)}=\inf\{\|u\|_{L^{q'}_tL^{r'}_x};\; (q,r)\,{\rm is}\, L^2\;
{\rm admissible}, \, 2\leq q\leq\infty,\; 2\leq r\leq6\},
$$
and
$$
\|u\|_{S'(\dot{H}^{-1/2})}=\inf\{\|u\|_{L^{q'}_tL^{r'}_x};\;(q,r)\,{\rm is}\,
\dot{H}^{-1/2}\; {\rm admissible},\, \left(\frac{4}{3}\right)^+\leq
q\leq2^-,\; 3^+\leq r\leq6^-\}.
$$

When a time interval $I\subset\R$ is given, we use ${S'(L^2;I)}$ to inform
that the temporal integral is evaluated over $I$, that is, we replace
$\|\cdot\|_{L^{q'}_tL^{r'}_x}$ by $\|\cdot\|_{L^{q'}_IL^{r'}_x}$. Similarly
to the other function spaces. For convenience, we define the space $S(L^2)\times S(L^2)$ endowed with the norm $\|(u,v)\|_{S(L^2)\times S(L^2)}\equiv \|u\|_{S(L^2)}+\|v\|_{S(L^2)}$. The spaces $S(\Hmeio)\times S(\Hmeio)$ and $S'(\dot{H}^{-1/2})\times S'(\dot{H}^{-1/2})$ are defined in the same way.

Next, we recall the well-known Strichartz inequalities.

\begin{lemma}\label{strilem}
With the above notation we have:
\begin{itemize}
    \item[(i)] (Linear estimates).
    $$
    \|e^{it\Delta}u_0\|_{S(L^2)}\lesssim\|u_0\|_{L^2},
    $$
    and
    $$
    \|e^{it\Delta}u_0\|_{S(\Hmeio)}\lesssim\|u_0\|_{\Hmeio}.
    $$
    \item[(ii)] (Inhomogeneous estimates).
    $$
   \left\| \int_0^te^{i(t-t')\Delta}f(\cdot,t')dt'\right\|_{S(L^2)}\lesssim\|f\|_{S'(L^2)},
   $$
   $$
     \left\| \int_0^te^{i(t-t')\Delta}f(\cdot,t')dt'\right\|_{S(\Hmeio)}\lesssim\|D^{1/2}f\|_{S'(L^2)},
   $$
   and
    $$
     \left\| \int_0^te^{i(t-t')\Delta}f(\cdot,t')dt'\right\|_{S(\Hmeio)}\lesssim\|f\|_{S'(\dot{H}^{-1/2})}.
   $$
\end{itemize}

\end{lemma}
\begin{proof}
See for instance \cite{caz} and \cite{hr}.
\end{proof}

The following decay estimate of the linear flow associated with the Schr\"odinger equation will also be useful in the sequel.
\begin{lemma}\label{Decay}
If $t\neq 0$, $1/p+1/p'=1$, and $p'\in [1,2]$ then
$$
e^{it\Delta}: L^{p'}(\R^3)\rightarrow L^{p}(\R^3)
$$
is continuous and we have the estimate
$$
\|e^{it\Delta}f\|_{L^p}\leq c |t|^{-\frac{3}{2}(1/p'-1/p)}\|f\|_{L^{p'}}.
$$
\end{lemma}
\begin{proof}
See \cite[Lemma 4.1]{LP15}.
\end{proof}

\section{Energy Inequalities and the Cauchy Problem}\label{sec3}

In this section we will provide the basic results concerning the Cauchy problem \eqref{3nls} we need. To begin with, we will establish some useful energy inequalities. First of all, recall the sharp Gagliardo-Nirenberg inequality.

\begin{lemma} \label{gnlemma}
For any $(u,v)\in H^1\times H^1$ there holds
\begin{equation}\label{GN}
\begin{split}
\|u\|^4_{L^4}+2\beta\|uv\|^2_{L^2}+\|v\|^4_{L^4}&\leq K_{GN} (\|u\|^2_{L^2}+\|v\|^2_{L^2})^{1/2}(\|\nabla u\|^2_{L^2}+\|\nabla v\|^2_{L^2})^{3/2}\\
&\equiv  K_{GN}M(u,v)^{1/2}A(u,v)^{3/2},
\end{split}
\end{equation}
where the sharp constant $K_{GN}$ is given by
\begin{equation}\label{sharpc}
K_{GN}=\frac{4}{3M(P,Q)^{1/2}A(P,Q)^{1/2}}
\end{equation}
and $(P,Q)$ is any ground state solution of the elliptic system \eqref{ellip1}.
\end{lemma}
\begin{proof}
See \cite[Section 3]{fm} and \cite[Proposition 2.3]{P2015}.
\end{proof}

\begin{remark}\label{remgn}
It can be shown that the sharp constant in \eqref{sharpc} is given by 
\begin{equation}\label{sharpc1}
K_{GN}=\dfrac{4}{3\sqrt3 m},
\end{equation}
where $m$ is the infimum, on the Nehari manifold, of the Lagrangian associated with \eqref{ellip1}. In particular, this shows that $K_{GN}$ does not depend on the choice of the ground state solution. Also, using a Pohozaev-type identity, it is not difficult to see that (see e.g., \cite{P2015})
\begin{equation}\label{sharpc2}
K_{GN}=\frac{4}{3\sqrt6 M(P,Q)^{1/2}E(P,Q)^{1/2}}.
\end{equation}
All together, relations \eqref{sharpc}-\eqref{sharpc2}, imply that the quantities $M(P,Q)E(P,Q)$ and also $M(P,Q)A(P,Q)$, appearing on the right-hand side of \eqref{EM} and \eqref{EM1}, do not depend on the choice of the ground state solution.
\end{remark}

\begin{lemma}\label{grolemma}
Assume that $(u,v)\in H^1\times H^1$ satisfies $M(u,v)A(u,v)\leq M(P,Q)A(P,Q)$. Then, the following statements  hold.
\begin{itemize}
\item[(i)] ${\displaystyle E(u,v)\geq \frac{1}{6}A(u,v).  }$
\item[(ii)]  ${\displaystyle M(u,v)A(u,v)\leq \frac{M(u,v)E(u,v)}{M(P,Q)E(P,Q)}M(P,Q)A(P,Q).  }$
\item[(iii)]  ${\displaystyle S(u,v)\geq 8A(u,v)\left( 1-\left(\frac{M(u,v)E(u,v)}{M(P,Q)E(P,Q)}\right)^{1/2} \right) }$, where
\begin{equation}\label{Ruv}
S(u,v)=8A(u,v)-6(\|u\|^4_{L^4}+2\beta\|uv\|^2_{L^2}+\|v\|^4_{L^4}).
\end{equation}
\end{itemize}
\end{lemma}
\begin{proof}
From Lemma \ref{gnlemma}, 
$$
E(u,v)\geq \frac{1}{2}A(u,v)-\frac{1}{4}K_{GN}M(u,v)^{1/2}A(u,v)^{3/2}.
$$
Multiplying this last inequality by $M(u,v)$ yields
\begin{equation}\label{en0}
\begin{split}
M(u,v)E(u,v)&\geq \frac{1}{2}\left(M(u,v)^{1/2}A(u,v)^{1/2} \right)^2-\frac{1}{4}K_{GN}\left(M(u,v)^{1/2}A(u,v)^{1/2} \right)^3\\
&\equiv f\left(M(u,v)^{1/2}A(u,v)^{1/2}\right),
\end{split}
\end{equation}
where by definition $f(x)=\frac{x^2}{2}-\frac{K_{GN}}{4}x^3$. In particular, since equality holds in \eqref{GN} for $(u,v)=(P,Q)$, we get
\begin{equation}\label{en1}
M(P,Q)E(P,Q)=f\left(M(P,Q)^{1/2}A(P,Q)^{1/2}\right).
\end{equation}
For $x>0$, function $f$ has a unique critical point at
$$
x_0=\frac{4}{3K_{GN}}=M(P,Q)^{1/2}A(P,Q)^{1/2}.
$$
It is easily seen that, on $(0,x_0)$,
\begin{equation}\label{en2}
f(x)\geq \frac{1}{6}x^2
\end{equation}
with equality for $x=x_0$, that is,
\begin{equation}\label{en3}
f\left(M(P,Q)^{1/2}A(P,Q)^{1/2}\right)= \frac{1}{6}M(P,Q)A(P,Q)
\end{equation}

In addition, $f$ is increasing on $(0,x_0)$ and decreasing on $(x_0,+\infty)$. Thus, since the hypothesis implies $M(u,v)^{1/2}A(u,v)^{1/2}\leq x_0$, we obtain
$$
M(u,v)E(u,v)\geq f\left(M(u,v)^{1/2}A(u,v)^{1/2}\right)\geq \frac{1}{6}\left( M(u,v)^{1/2}A(u,v)^{1/2} \right)^2.
$$
This proves part (i) of the lemma. From \eqref{en3} and \eqref{en1}, we have
$$
M(P,Q)E(P,Q)=\frac{1}{6}M(P,Q)A(P,Q),
$$
which together with \eqref{en0} and \eqref{en2} imply part (ii). Finally, we note from Lemma \ref{gnlemma} and the definition  of $x_0$ that
\[
\begin{split}
\frac{1}{8}M(u,v)S(u,v)&\geq M(u,v)A(u,v)-\frac{3}{4}K_{GN}M(u,v)^{3/2}A(u,v)^{3/2}\\
&=M(u,v)A(u,v)\left( 1-\frac{1}{x_0} M(u,v)^{1/2}A(u,v)^{1/2}  \right)\\
&=M(u,v)A(u,v)\left( 1-\frac{M(u,v)^{1/2}A(u,v)^{1/2}}{M(P,Q)^{1/2}A(P,Q)^{1/2}} \right).
\end{split}
\]
This establishes part (iii) and completes the proof of Lemma \ref{grolemma}.
\end{proof}

In view of the above lemma and the global result stated in Theorem \ref{globalPastor} we have the following small data global theory in $H^1\times H^1$.

\begin{corollary}[Small data global theory in $H^1\times H^1$] \label{teocri2} There exists $\delta_{sd}>0$ such
that if $(u_0,v_0)\in H^1\times H^1$ satisfies
$$
\|(u_0,v_0)\|_{H^1\times H^1} \leq \delta_{sd}
$$ 
then the solution $(u,v)$ of the initial value problem \eqref{3nls} is globally defined in $H^1\times H^1$. In addition there exists $C_{sd}>0$ such that
\begin{equation}\label{uvH1}
\|(u,v)\|_{L_t^{\infty}(H^1\times H^1)}\leq C_{sd}\delta_{sd}.
\end{equation}
\end{corollary}
\begin{proof}
For simplicity in this proof we use the equivalent norm in $H^1$ given by
$$
\|u_0\|^2_{H^1}=\|u_0\|^2_{L^2}+\|\nabla u_0\|^2_{L^2}.
$$
Therefore,
\begin{equation}\label{MAuv}
\|(u_0,v_0)\|^2_{H^1\times H^1}=M(u_0,v_0)+A(u_0,v_0).
\end{equation}
Furthermore, it is easy to see that 
$$
M(u_0,v_0)A(u_0,v_0)\leq \|(u_0,v_0)\|^2_{H^1\times H^1}
$$
and
$$
M(u_0,v_0)E(u_0,v_0)\leq \frac12M(u_0,v_0)A(u_0,v_0).
$$
So, there exists $\delta_{sd}>0$ sufficiently small such that if $\|(u_0,v_0)\|_{H^1\times H^1} \leq \delta_{sd}$ then the relations \eqref{EM} and \eqref{EM1} are satisfied. In view of Theorem \ref{globalPastor}, the solution $(u,v)$ of the initial value problem \eqref{3nls} is global and, since $M(u,v)$ and $E(u,v)$ are conserved quantities,
$$
M(u(t),v(t))E(u(t),v(t))\leq \frac12\delta_{sd}^2
$$
and
$$
M(u(t),v(t))A(u(t),v(t))<M(P,Q)A(P,Q).
$$
Applying Lemma \ref{grolemma} (ii) there exists a constant $C_{sd}>0$ such that
\begin{equation}\label{MAuvt}
M(u(t),v(t))A(u(t),v(t))<C^2_{sd}\delta_{sd}^2.
\end{equation}

Finally, since the mass $M(u(t),v(t))$ is preserved, we deduce (possibly changing the constant $C_{sd}$) \eqref{uvH1} from \eqref{MAuv} and \eqref{MAuvt}. This completes the proof of Corollary \ref{teocri2}. 
\end{proof}

Next we recall the small data global theory in the critical Sobolev space $\Hmeio\times \Hmeio$, the $H^1$-scattering criterion and the existence of wave operators obtained by the second author in \cite{P2015}.

\begin{theorem}[Small data global theory in $\Hmeio\times \Hmeio$] \label{teocri} Let $A>0$ and assume that
$(u_0,v_0)\in \Hmeio\times \Hmeio$ satisfies
$\|(u_0,v_0)\|_{\Hmeio\times \Hmeio}\leq A$. There is $\delta>0$ such
that if 
$$
\|(e^{it\Delta}u_0, e^{it\Delta}v_0)\|_{S(\Hmeio)\times S(\Hmeio)}\leq \delta,
$$ 
then the
initial value problem \eqref{3nls} is globally well-posed in
$\Hmeio\times \Hmeio$. In addition there exists $c>0$ such that
$$
\|(u,v)\|_{S(\Hmeio)\times S(\Hmeio)}\leq
2\|(e^{it\Delta}u_0, e^{it\Delta}v_0)\|_{S(\Hmeio)\times S(\Hmeio)}
$$
and
$$
\|(D^{1/2}u, D^{1/2}v)\|_{S(L^2)\times S(L^2)}\leq
2c\|(u_0,v_0)\|_{\Hmeio\times \Hmeio}.
$$
\end{theorem}
\begin{proof}
See \cite[Theorem 4.1]{P2015}.
\end{proof}

The next result gives a sufficient condition for a uniformly bounded global  solution in $H^1\times H^1$ scatters.

\begin{theorem}[$H^1$ scattering]\label{h1scat}
Let $(u_0,v_0)\in H^1\times H^1$. Suppose that the initial value problem \eqref{3nls} is globally well-posed in $H^1\times H^1$ with 
$B:=\displaystyle{\sup_{t\in \R}}\|(u(t),v(t))\|_{H^1\times H^1}<\infty$ and $K:=\|u\|_{L^5_{x,t}}+\|v\|_{L^5_{x,t}}<\infty$. Then the solution $(u(t), v(t))$ scatters forward and backward in time in the sense of Definition \ref{scatdef}.
\end{theorem}
\begin{proof}
Assuming that $\displaystyle{\sup_{t\in [0,+\infty)}}\|(u(t),v(t))\|_{H^1\times H^1}<\infty$ the forward scattering was proved in \cite[Theorem 1.3]{P2015}. The same argument can be applied to prove backward scattering if we suppose $\displaystyle{\sup_{t\in (-\infty,0]}}\|(u(t),v(t))\|_{H^1\times H^1}<\infty$.
\end{proof}

Note  that $(5,5)$ is an $\dot{H}^{1/2}$ admissible pair. A simple inspection in the proof of Theorem 1.3 in \cite{P2015} reveals that it still holds if we replace $(5,5)$  by any $\dot{H}^{1/2}$ admissible pair. In particular, a sufficient condition for scattering is the finiteness of the quantity $\|(u,v)\|_{S(\Hmeio\times\Hmeio)}$.

\begin{theorem}[Existence of wave operators]\label{waveop}
Assume $\phi^+,\psi^+\in H^1$ and
\begin{equation}\label{phipsi+}
\frac{1}{2}M(\phi^+,\psi^+)A(\phi^+,\psi^+)<M(P,Q)E(P,Q).
\end{equation}
Then, there exists $(u_0,v_0)\in H^1\times H^1$ such that the solution
$(u(t),v(t))$ of \eqref{3nls} with initial condition $(u_0,v_0)$ exists globally and satisfies
\begin{equation*}
\lim_{t\to+\infty}\|(u(t),v(t))-(e^{it\Delta}\phi^+,e^{it\Delta}\psi^+)\|_{H^1\times H^1}=0,
\end{equation*}
In addition,
\begin{equation}\label{waveop1}
M(u_0,v_0)A(u(t),v(t))<M(P,Q)A(P,Q)
\end{equation}
and
\begin{equation}\label{waveop2}
M(u(t),v(t))=M(\phi^+,\psi^+), \quad E(u(t),v(t))=\frac{1}{2}A(\phi^+,\psi^+).
\end{equation}
\end{theorem}
\begin{proof}
See \cite[Theorem 1.4]{P2015}.
\end{proof}
\begin{remark}
A similar result backward in time also holds assuming that $\phi^-,\psi^-\in H^1$ satisfy \eqref{phipsi+} and the limit is taken as $t\rightarrow -\infty$. 
\end{remark}

Next, we prove a so called long time  perturbation result. This will be necessary in the construction of the critical solution below.

\begin{proposition}[Long time perturbation theory]\label{LTstabteo}
Given any $A>0$, there are $\varepsilon>0$ and $C>0$, depending only on $A$, such that the following statement holds. Assume that $(u,v)\in C([t_0,\infty);H^1\times H^1)$ is a solution of \eqref{3nls} and $(\widetilde{u},\widetilde{v})\in C([t_0,\infty);H^1\times H^1)$ satisfy
\begin{equation}\label{3nlstil}
\begin{cases}
i\partial_t\widetilde{u}+\Delta \widetilde{u}+(|\widetilde{u}|^2+\beta |\widetilde{v}|^2)\widetilde{u}=e_1,\\
i\partial_t\widetilde{v}+\Delta \widetilde{v}+(|\widetilde{v}|^2+\beta |\widetilde{u}|^2)\widetilde{v}=e_2,
\end{cases}
\end{equation}
for suitable functions $e_1$ and $e_2$. If
\begin{equation}\label{st1}
\|(\widetilde{u},\widetilde{v})\|_{S(\Hmeio)\times S(\Hmeio)}\leq A, \quad \|(e_1, e_2)\|_{S'(\dot{H}^{-1/2})\times S'(\dot{H}^{-1/2})}\leq \varepsilon,
\end{equation}
\begin{equation}\label{st2}
\|(e^{i(t-t_0)\Delta}(u(t_0)-\widetilde{u}(t_0)), e^{i(t-t_0)\Delta}(v(t_0)-\widetilde{v}(t_0)))\|_{S(\Hmeio)\times S(\Hmeio)}\leq  \varepsilon,
\end{equation}
then
$$
\|(u-\widetilde{u}, v-\widetilde{v})\|_{S(\Hmeio)\times S(\Hmeio)}\leq C \varepsilon.
$$
\end{proposition}
\begin{proof}
Define $w=u-\widetilde{u}$ and $z=v-\widetilde{v}$. Then, $w$ and $z$ satisfy
\begin{equation}\label{3nlss}
\begin{cases}
i\partial_tw+\Delta w+f_1+e_1=0,\\
i\partial_tz+\Delta z+f_2+e_2=0,
\end{cases}
\end{equation}
with
$$
f_1=|u|^2u+\beta |v|^2u-|\widetilde{u}|^2\widetilde{u}-\beta |\widetilde{v}|^2\widetilde{u}, \quad f_2=|v|^2v+\beta |u|^2v-|\widetilde{v}|^2\widetilde{v}-\beta |\widetilde{u}|^2\widetilde{v}.
$$
Now we can take a partition of the interval $[t_0,\infty)$ into $N$ subintervals (with the number $N$ depending only on $A$) of the form $I_j=[t_j,t_{j+1}]$, $j=0,\ldots,N-1$, such that$ \|\widetilde{u}\|_{S(\Hmeio;I_j)}+\|\widetilde{v}\|_{S(\Hmeio;I_j)}\leq\delta$, where $\delta>0$ will  be appropriately chosen later.
Since 
$$
\|w\|_{S(\Hmeio)}+\|z\|_{S(\Hmeio)}\leq \sum_{j=0}^{N-1}\big(\|w\|_{S(\Hmeio;I_j)}+\|z\|_{S(\Hmeio;I_j)}\big),
$$
it is sufficient to show that, for $j=0,\ldots,N-1$,
\begin{equation}\label{st3}
\|w\|_{S(\Hmeio;I_j)}+\|z\|_{S(\Hmeio;I_j)}\leq C\varepsilon.
\end{equation}

The system \eqref{3nlss} with initial time $t_j$ can be converted into the integral equations
\begin{equation}\label{st3.1}
w(t)=e^{i(t-t_j)\Delta}w(t_j)+i\int_{t_j}^t e^{i(t-s)\Delta}(f_1(s)+e_1(s))ds,
\end{equation}
\begin{equation}\label{st3.2}
z(t)=e^{i(t-t_j)\Delta}z(t_j)+i\int_{t_j}^t e^{i(t-s)\Delta}(f_2(s)+e_2(s))ds.
\end{equation}
Thus applying Strichartz estimates,
\begin{equation}\label{st4}
\|w\|_{S(\Hmeio;I_j)}\leq \|e^{i(t-t_j)\Delta}w(t_j)\|_{S(\Hmeio;I_j)}+\|f_1\|_{L_{I_j}^{10/3}L_x^{5/4}}+\|e_1\|_{S(\Hmeio;I_j)}
\end{equation}
\begin{equation}\label{st5}
\|z\|_{S(\Hmeio;I_j)}\leq \|e^{i(t-t_j)\Delta}z(t_j)\|_{S(\Hmeio;I_j)}+\|f_2\|_{L_{I_j}^{10/3}L_x^{5/4}}+\|e_2\|_{S(\Hmeio;I_j)},
\end{equation}
where we have used that $(10/7,5)$ is an $\dot{H}^{-1/2}$ admissible pair. Observe that
$$
|f_1|\leq (|w|^2+\beta|z|^2)|w|+(|w|^2+\beta|z|^2)|\widetilde{u}|+(|\widetilde{u}|^2+\beta|\widetilde{v}|^2)|w|
$$
and
$$
|f_2|\leq (|z|^2+\beta|w|^2)|z|+(|z|^2+\beta|w|^2)|\widetilde{v}|+(|\widetilde{v}|^2+\beta|\widetilde{u}|^2)|z|.
$$
Therefore, H\"older's inequality yields
\begin{equation}\label{st6}
\begin{split}
\|f_1\|_{L_{I_j}^{10/3}L_x^{5/4}}&\leq \Big(\|w\|_{L_{I_j}^{20}L_x^{10/3}}^2+\beta  \|z\|_{L_{I_j}^{20}L_x^{10/3}}^2 \Big)\|w\|_{L_{I_j}^{5}L_x^{5}}+\Big(\|w\|_{L_{I_j}^{20}L_x^{10/3}}^2 \\
&\;\;\;\;+\beta  \|z\|_{L_{I_j}^{20}L_x^{10/3}}^2 \Big)\|\widetilde{u}\|_{L_{I_j}^{5}L_x^{5}}+\Big(\|\widetilde{u}\|_{L_{I_j}^{20}L_x^{10/3}}^2+\beta  \|\widetilde{v}\|_{L_{I_j}^{20}L_x^{10/3}}^2 \Big)\|w\|_{L_{I_j}^{5}L_x^{5}}\\
&\lesssim \|w\|^3_{S(\Hmeio;I_j)}+\|z\|^3_{S(\Hmeio;I_j)}+\delta^2\|w\|_{S(\Hmeio;I_j)}\\
&\;\;\;\;+\delta\Big( \|w\|^2_{S(\Hmeio;I_j)}+\|z\|^2_{S(\Hmeio;I_j)} \Big)\\
&\lesssim \Big(\|w\|_{S(\Hmeio;I_j)}+\|z\|_{S(\Hmeio;I_j)}\Big)^3+\delta^2\|w\|_{S(\Hmeio;I_j)}\\
&\quad+\delta\Big(\|w\|_{S(\Hmeio;I_j)}+\|z\|_{S(\Hmeio;I_j)}\Big)^2,
\end{split}
\end{equation}
where we used that $(5,5)$ and $(20,10/3)$ are $\Hmeio$ admissible pairs.  In a similar fashion,
\begin{equation}\label{st7}
\begin{split}
\|f_2\|_{L_{I_j}^{10/3}L_x^{5/4}}
&\lesssim \Big(\|w\|_{S(\Hmeio;I_j)}+\|z\|_{S(\Hmeio;I_j)}\Big)^3+\delta^2\|z\|_{S(\Hmeio;I_j)}\\
&\quad+\delta\Big(\|w\|_{S(\Hmeio;I_j)}+\|z\|_{S(\Hmeio;I_j)}\Big)^2.
\end{split}
\end{equation}
By summing equations \eqref{st4} and \eqref{st5}, using \eqref{st6}, \eqref{st7} and the assumption, we deduce the existence of a large constant $c>0$ such that
\begin{equation}\label{st8}
\begin{split}
\|w\|_{S(\Hmeio;I_j)}&+\|z\|_{S(\Hmeio;I_j)}\leq 
A_j+c\varepsilon+c\Big(\|w\|_{S(\Hmeio;I_j)}+\|z\|_{S(\Hmeio;I_j)}\Big)^3\\
&+c\delta^2(\|w\|_{S(\Hmeio;I_j)}+\|z\|_{S(\Hmeio;I_j)})
\quad+c\delta\Big(\|w\|_{S(\Hmeio;I_j)}+\|z\|_{S(\Hmeio;I_j)}\Big)^2,
\end{split}
\end{equation}
with
$$
A_j:=\|e^{i(t-t_j)\Delta}w(t_j)\|_{S(\Hmeio;I_j)}+\|e^{i(t-t_j)\Delta}z(t_j)\|_{S(\Hmeio;I_j)}.
$$

On the other hand, recall that solving \eqref{3nlss} is equivalent to $(\Phi(w),\Psi(z))=(w,z)$, where $\Phi(w)$ and $\Psi(z)$ are the right-hand side of \eqref{st3.1} and \eqref{st3.2}, respectively. Thus, in view of \eqref{st8}, one sees 
\begin{equation}\label{st8.1}
\begin{split}
\|\Phi(w)\|_{S(\Hmeio;I_j)}+\|\Psi(z)\|_{S(\Hmeio;I_j)}\leq &A_j+c\varepsilon+c\Big(\|w\|_{S(\Hmeio;I_j)}+\|z\|_{S(\Hmeio;I_j)}\Big)^3\\
&+c\delta^2(\|w\|_{S(\Hmeio;I_j)}+\|z\|_{S(\Hmeio;I_j)})\\
&+c\delta\Big(\|w\|_{S(\Hmeio;I_j)}+\|z\|_{S(\Hmeio;I_j)}\Big)^2.
\end{split}
\end{equation}
Define
$$
B_a=\{w,z \;\;\mbox{on} \;\;\R^2: \|w\|_{S(\Hmeio;I_j)}+\|z\|_{S(\Hmeio;I_j)}\leq a:=2(A_j+c\varepsilon)\}.
$$

Assume that
\begin{equation}\label{st9}
\delta\leq \frac{1}{\sqrt{6c}}\qquad \mbox{and}\qquad 2(A_j+c\varepsilon)\leq \frac{1}{\sqrt{6c}}.
\end{equation}
With these choice, it is easy to see that the right-hand side of \eqref{st8.1} is bounded by $a$ provided $(w,z)$ belongs to $B_a$. The fixed point theorem then gives
\begin{equation}
\|(w,z)\|_{S(\Hmeio;I_j)\times S(\Hmeio;I_j)}\leq  2\|(e^{i(t-t_j)\Delta}w(t_j),e^{i(t-t_j)\Delta}z(t_j))\|_{S(\Hmeio;I_j)\times S(\Hmeio;I_j)}+2c\varepsilon
\end{equation}

Now we proceed as follows. By taking $t=t_{j+1}$ in the integral equations above and applying $e^{i(t-t_{j+1})\Delta}$, we obtain
$$
e^{i(t-t_{j+1})\Delta}w(t_{j+1})=e^{i(t-t_{j})\Delta}w(t_j)+i\int_{t_j}^{t_{j+1}} e^{i(t-s)\Delta}(f_1(s)+e_1(s))ds,
$$
$$
e^{i(t-t_{j+1})\Delta}z(t_{j+1})=e^{i(t-t_{j})\Delta}z(t_j)+i\int_{t_j}^{t_{j+1}} e^{i(t-s)\Delta}(f_2(s)+e_2(s))ds.
$$
Using that the integral part is confined to the interval $I_j$, one can repeat the above arguments to show that
\begin{equation}\label{st11}
\begin{split}
\|e^{i(t-t_{j+1})\Delta}w(t_{j+1})&\|_{S(\Hmeio)}+\|e^{i(t-t_{j+1})\Delta}z(t_{j+1})\|_{S(\Hmeio)} \\
&\leq \|e^{i(t-t_j)\Delta}w(t_j)\|_{S(\Hmeio)}+\|e^{i(t-t_j)\Delta}z(t_j)\|_{S(\Hmeio)}\\
&\quad+c\varepsilon+c\Big(\|w\|_{S(\Hmeio;I_j)}+\|z\|_{S(\Hmeio;I_j)}\Big)^3+\delta^3.
\end{split}
\end{equation}
Provided
\begin{equation}\label{st12}
\|e^{i(t-t_j)\Delta}w(t_j)\|_{S(\Hmeio)}+\|e^{i(t-t_j)\Delta}z(t_j)\|_{S(\Hmeio)}+c\varepsilon\leq\frac{1}{2\sqrt{6c}},
\end{equation}
we then get
\begin{equation}\label{st13}
\begin{split}
\|e^{i(t-t_{j+1})\Delta}w(t_{j+1})&\|_{S(\Hmeio)}+\|e^{i(t-t_{j+1})\Delta}z(t_{j+1})\|_{S(\Hmeio)} \\
&\leq 2\|e^{i(t-t_j)\Delta}w(t_j)\|_{S(\Hmeio)}+2\|e^{i(t-t_j)\Delta}z(t_j)\|_{S(\Hmeio)}+2c\varepsilon.
\end{split}
\end{equation}
Since \eqref{st12} implies the second assumption in \eqref{st9} it suffices to show \eqref{st12}. Using \eqref{st13} and iterating, one deduces
\[
\begin{split}
\|e^{i(t-t_{j})\Delta}w(t_{j})&\|_{S(\Hmeio)}+\|e^{i(t-t_{j})\Delta}z(t_{j})\|_{S(\Hmeio)} \\
&\leq 2^j\Big(\|e^{i(t-t_0)\Delta}w(t_0)\|_{S(\Hmeio)}+\|e^{i(t-t_0)\Delta}z(t_0)\|_{S(\Hmeio)}\Big)+c\varepsilon\sum_{k=1}^j2^k\\
&\leq 2^j\varepsilon+2^{j+1}c\varepsilon\\
&\leq 2^{j+2}c\varepsilon.
\end{split}
\]
Hence, it suffices to choose $\varepsilon >0$ satisfying 
$$2^{N+2}c\varepsilon+c\varepsilon\leq\frac{1}{2\sqrt{6c}}.
$$
The proof is thus completed.
\end{proof}

\section{Profile and Energy decomposition}\label{sec4}

In this section we first prove a profile expansion for two bounded sequences in $H^1(\R^3)$ in the spirit of Duyckaerts-Holmer-Roudenko \cite{dhr} (see also Keraani \cite{Ke01}). The main point in our result is that the space and time shifts, $x_n^j$ and $t_n^j$, are the same for both sequences. We point out that our proof can be adapted to obtain a similar result for any finite quantities of bounded sequences.

\begin{theorem}[Profile Decomposition]\label{profdec}
Let $\{\phi_n\}$ and  $\{\widetilde{\phi}_n\}$ be two bounded sequences in $H^1(\R^3)$. Then, for each $M\geq1$, there exist subsequences of  $\{\phi_n\}$ and  $\{\widetilde{\phi}_n\}$, still denoted by  $\{\phi_n\}$ and  $\{\widetilde{\phi}_n\}$, respectively, and 
\begin{itemize}
\item[(i)] for each $1\leq j\leq M$, there are profiles  $\{\psi^j\}$ and  $\{\widetilde{\psi}^j\}$ in $H^1$,
\item[(ii)] for each $1\leq j\leq M$, there exists a sequence of time shifts $\{t_n^j\}$ in $\R$,
\item[(iii)] for each $1\leq j\leq M$, there exists a sequence of space shifts $\{x_n^j\}$ in $\R^3$,
\item[(iv)] there are two sequences of remainders $\{W_n^M\}$ and $\{\widetilde{W}_n^M\}$  in $H^1$,
\end{itemize}
such that
$$
(\phi_n(x), \wtphi_n(x))=\sum_{j=1}^M(e^{-it_n^j\Delta}\psi^j(x-x_n^j), e^{-it_n^j\Delta}\wtpsi^j(x-x_n^j))+(W_n^M(x), \wtw_n^M(x)).
$$
For $1\leq j\neq k\leq M$, the time and space sequences satisfy
\begin{equation}\label{pr1}
\lim_{n\to\infty} (|t_n^j-t_n^k|+|x_n^j-x_n^k|)=\infty.
\end{equation}
The remainder sequences satisfy
\begin{equation}\label{pr2}
\lim_{M\to\infty}\Big[\lim_{n\to\infty}\big( \|(e^{it\Delta}W_n^M, e^{it\Delta}\wtw_n^M)\|_{S(\Hmeio)\times S(\Hmeio)}\Big)\Big]=0.
\end{equation}
In addition, for $s\in[0,1]$ and any $M\geq1$, we have the asymptotic expansions
\begin{equation}\label{pr3}
\|\phi_n\|^2_{\dot{H}^s}=\sum_{j=1}^M\|\psi^j\|^2_{\dot{H}^s}+\|W_n^M\|^2_{\dot{H}^s}+o_n(1)
\end{equation}
and
\begin{equation}\label{pr4}
\|\wtphi_n\|^2_{\dot{H}^s}=\sum_{j=1}^M\|\wtpsi^j\|^2_{\dot{H}^s}+\|\wtw_n^M\|^2_{\dot{H}^s}+o_n(1).
\end{equation}
\end{theorem}
\begin{proof}
The proof is similar to that of Lemma 1 in \cite{dhr}. The constructions are performed into two main steps.

\noindent{\bf Step 1.} {\it Construction of $\psi^1$ and $\wtpsi^1$}. Let us start by defining
$$
A_1=\limsup_{n\rightarrow \infty}(\|e^{it\Delta}\phi_n\|_{L^\infty_tL^3_x}+\|e^{it\Delta}\wtphi_n\|_{L^\infty_tL^3_x}).
$$
If $(q,r)$ is any $\Hmeio$ admissible pair then, by interpolation,
\begin{equation}\label{pr4.1}
\|e^{it\Delta}\phi_n\|_{L^q_tL^r_x}\leq \|e^{it\Delta}\phi_n\|_{L^4_tL^6_x}^\theta \|e^{it\Delta}\phi_n\|_{L^\infty_tL^3_x}^{1-\theta},
\end{equation}
where $\theta=4/q\in(0,1)$. Since $(4,6)$ is $\Hmeio$ admissible, Strichartz's estimates combined with the boundedness of $\{\phi_n\}$ in $\Hmeio$ imply that $\|e^{it\Delta}\phi_n\|_{L^q_tL^r_x}\leq C \|e^{it\Delta}\phi_n\|_{L^\infty_tL^3_x}^{1-\theta}$. Thus, if $A_1=0$ we promptly deduce that 
$$
\limsup_{n\rightarrow \infty}\|e^{it\Delta}\phi_n\|_{S(\Hmeio)}=0
$$ 
and (similarly) 
$$
\limsup_{n\rightarrow \infty}\|e^{it\Delta}\wtphi_n\|_{S(\Hmeio)}=0.
$$
Consequently, we may take $\psi^j=\wtpsi^j=0$, for all $1\leq j\leq M$.

Assume now $A_1>0$ and define
$$
c_1=\limsup_{n\rightarrow \infty}(\|\phi_n\|_{H^1}+\|\wtphi_n\|_{H^1}).
$$
Let $r>0$ be a real number such that $\frac{c_1}{\sqrt{r}}=\frac{A_1}{8}$. Take $\chi$ to be a real Schwartz function satisfying  $\widehat{\chi}(\xi)=1$ for $\frac{1}{r}\leq|\xi|\leq r$ and ${\rm supp}(\widehat{\chi})\subset\left[\frac{1}{2r},2r\right]$. Using the Sobolev embedding $\Hmeio(\R^3)\hookrightarrow L^3(\R^3)$ and the definition of $\chi$, it is not difficult to see that (see, for instance, proof of Lemma 5.2 in \cite{hr})
$$
\|e^{it\Delta}\phi_n-\chi\ast e^{it\Delta}\phi_n\|_{L^\infty_tL^3_x}^2\leq \frac{c_1^2}{r}
$$
and
$$
\|e^{it\Delta}\wtphi_n-\chi\ast e^{it\Delta}\wtphi_n\|_{L^\infty_tL^3_x}^2\leq \frac{c_1^2}{r}.
$$
These last two inequalities imply that
\begin{equation}\label{pr5}
\|\chi\ast e^{it\Delta}\phi_n\|_{L^\infty_tL^3_x}+ \|\chi\ast e^{it\Delta}\wtphi_n\|_{L^\infty_tL^3_x}\geq  \|e^{it\Delta}\phi_n\|_{L^\infty_tL^3_x} + \|e^{it\Delta}\wtphi_n\|_{L^\infty_tL^3_x}-\frac{2c_1}{\sqrt{r}}.
\end{equation}

By the definitions of $A_1$ and $r$, we can assume that the right-hand side of \eqref{pr5} is larger than $\frac{3A_1}{4}-\frac{A_1}{4}=\frac{A_1}{2}$, that is,
\begin{equation}\label{pr6}
\|\chi\ast e^{it\Delta}\phi_n\|_{L^\infty_tL^3_x}+ \|\chi\ast e^{it\Delta}\wtphi_n\|_{L^\infty_tL^3_x}\geq  \frac{A_1}{2}.
\end{equation}
Now, the standard  interpolation for Lebesgue norm and the fact that $\{e^{it\Delta}\}$ is a unitary group in $L^2$ imply
\begin{equation}\label{pr7}
\begin{split}
\|\chi\ast e^{it\Delta}\phi_n\|_{L^\infty_tL^3_x}^3&\leq \|\chi\ast e^{it\Delta}\phi_n\|_{L^\infty_tL^2_x}^2\|\chi\ast e^{it\Delta}\phi_n\|_{L^\infty_tL^\infty_x}\\
& \leq \|\phi_n\|_{L^2}^2\|\chi\ast e^{it\Delta}\phi_n\|_{L^\infty_tL^\infty_x}\\
& \leq c_1^2\|\chi\ast e^{it\Delta}\phi_n\|_{L^\infty_tL^\infty_x}.
\end{split}
\end{equation}
Hence,
\begin{equation}\label{pr8}
\begin{split}
\Big(\|\chi\ast e^{it\Delta}\phi_n\|_{L^\infty_tL^3_x}+&\|\chi\ast e^{it\Delta}\wtphi_n\|_{L^\infty_tL^3_x}\Big)^3\\
& \leq 4\Big(\|\chi\ast e^{it\Delta}\phi_n\|_{L^\infty_tL^3_x}^3
 +\|\chi\ast e^{it\Delta}\wtphi_n\|_{L^\infty_tL^3_x}^3\Big)\\
&\leq 4c_1^2\Big(\|\chi\ast e^{it\Delta}\phi_n\|_{L^\infty_tL^\infty_x}
+\|\chi\ast e^{it\Delta}\wtphi_n\|_{L^\infty_tL^\infty_x}\Big).
\end{split}
\end{equation}
Inequalities \eqref{pr6} and \eqref{pr8} yield
$$
\|\chi\ast e^{it\Delta}\phi_n\|_{L^\infty_tL^\infty_x}+ \|\chi\ast e^{it\Delta}\wtphi_n\|_{L^\infty_tL^\infty_x} \geq \frac{A_1^3}{32c_1^2}.
$$
From the last estimate it is clear that
$$
\max\left\{\|\chi\ast e^{it\Delta}\phi_n\|_{L^\infty_tL^\infty_x}, \|\chi\ast e^{it\Delta}\wtphi_n\|_{L^\infty_tL^\infty_x}\right\}\geq \frac{A_1^3}{64c_1^2}.
$$
Now we can take, for each $n$, $t_n^1\in \R$ and $x_n^1\in\R^3$ such that
$$
\max\left\{|\chi\ast e^{it_n^1\Delta}\phi_n(x_n^1)|, |\chi\ast e^{it_n^1\Delta}\wtphi_n(x_n^1)|\right\} \geq \frac{A_1^3}{128c_1^2},
$$
which implies
\begin{equation}\label{pr9}
|\chi\ast e^{it_n^1\Delta}\phi_n(x_n^1)|+ |\chi\ast e^{it_n^1\Delta}\wtphi_n(x_n^1)| \geq \frac{A_1^3}{128c_1^2}.
\end{equation}

Since $\{ e^{it_n^1\Delta}\phi_n(\cdot+x_n^1)\}$ and $\{ e^{it_n^1\Delta}\wtphi_n(\cdot+x_n^1)\}$ are bounded sequences in the Hilbert space $H^1$, there exist $\psi^1$ and $\wtpsi^1$ such that (up to a subsequence)
$$
e^{it_n^1\Delta}\phi_n(\cdot+x_n^1)\rightharpoonup \psi^1 \quad \mbox{and}\quad e^{it_n^1\Delta}\wtphi_n(\cdot+x_n^1)\rightharpoonup\wtpsi^1,
$$
in $H^1$. By definition of convolution and weak convergence, we deduce from \eqref{pr9},
$$
\left| \int_{\R^3}\chi(y)\psi^1(-y)dy\right|+\left| \int_{\R^3}\chi(y)\wtpsi^1(-y)dy\right|\geq  \frac{A_1^3}{64c_1^2}.
$$
H\"older's inequality and Plancherel's identity give
$$
\|\chi\|_{\dot{H}^{-1/2}}\|\psi^1\|_{\Hmeio}+\|\chi\|_{\dot{H}^{-1/2}}\|\wtpsi^1\|_{\Hmeio}\geq \frac{A_1^3}{64c_1^2}.
$$
Taking note that $\|\chi\|_{\dot{H}^{-1/2}}\leq r$, we deduce the existence of a constant $K>0$ such that
\begin{equation}\label{pr10}
\|\psi^1\|_{\Hmeio}+\|\wtpsi^1\|_{\Hmeio}\geq \frac{A_1^5}{Kc_1^4}.
\end{equation}
Define $W_n^1(x)=\phi_n(x)-e^{-it_n^1\Delta}\psi^1(x-x_n^1)$ and $\wtw_n^1(x)=\wtphi_n(x)-e^{-it_n^1\Delta}\wtpsi^1(x-x_n^1)$. Using the definition of $W_n^1$ and the weak convergence, it is easy to see that, for any $s\in[0,1]$,
$$
\|W_n^1\|^2_{\dot{H}^s}=\|e^{it_n^1\Delta}\phi_n(\cdot+x_n^1)\|^2_{\dot{H}^s}-\|\psi^1\|^2_{\dot{H}^s}+o_n(1).
$$
The group property immediately gives \eqref{pr3} in the case $M=1$. The expansion \eqref{pr4} is proved in a similar fashion.

\noindent{\bf Step 2.} {\it Construction of $\psi^j$ and $\wtpsi^j$}. The construction is by induction. Let $M\geq2$ and assume we have constructed the sequences (in $n$) $\{t_n^j\}$, $\{x_n^j\}$, and the functions $\psi^j$ and $\wtpsi^j$ for $j\in\{1,\ldots,M-1\}$. Note that by construction (see \eqref{pr3} and \eqref{pr4}) the sequences $\{W_n^j\}$ and $\{\wtw_n^j\}$ are uniformly bounded in $H^1$, for any $j\in\{1,\ldots,M-1\}$. Let us define
$$
A_M=\limsup_{n\rightarrow \infty}(\|e^{it\Delta}W_n^{M-1}\|_{L^\infty_tL^3_x}+\|e^{it\Delta}\wtw_n^{M-1}\|_{L^\infty_tL^3_x}).
$$
If $A_M=0$, as in Step 1, we may take $\psi^j=\wtpsi^j=0$ for all $j\geq M$. If $A_M>0$, we apply Step 1 to the sequences $\{W_n^{M-1}\}$ and $\{\wtw_n^{M-1}\}$. Thus, up to subsequences, we find $\{t_n^M\}$, $\{x_n^M\}$ and functions $\psi^M,\wtpsi^M\in H^1$ such that
\begin{equation}\label{pr11}
e^{it_n^M\Delta}W_n^{M-1}(\cdot+x_n^M)\rightharpoonup \psi^M, \quad  e^{it_n^M\Delta}\wtw_n^{M-1}(\cdot+x_n^M)\rightharpoonup\wtpsi^M,
\end{equation}
and (see \eqref{pr10})
\begin{equation}\label{pr12}
\|\psi^M\|_{\Hmeio}+\|\wtpsi^M\|_{\Hmeio}\geq \frac{A_M^5}{Kc_M^4},
\end{equation}
where $c_M=\limsup_{n\rightarrow \infty}(\|W_n^{M-1}\|_{H^1}+\|\wtw_n^{M-1}\|_{H^1})$. As in Step 1, we define
$$
W_n^M(x)=W_n^{M-1}(x)-e^{-it_n^M\Delta}\psi^M(x-x_n^M), \quad
\wtw_n^M(x)=\wtw_n^{M-1}(x)-e^{-it_n^M\Delta}\wtpsi^M(x-x_n^M).
$$
The weak convergence implies
$$
\|W_n^M\|_{\dot{H}^s}=\|W_n^{M-1}\|_{\dot{H}^s}-\|\psi^M\|_{\dot{H}^s}+o_n(1)
$$
The induction assumption then yields \eqref{pr3} at rank $M$. One proves \eqref{pr4} in a similar manner.

In order to show that \eqref{pr1} holds we need the following lemma.

\begin{lemma} \label{cazlem}
Let $\{t_n\}\subset\R$ and $\{x_n\}\subset\R^3$ be two sequences.
\begin{itemize}
\item[(i)] If
\begin{equation}\label{pr13}
\lim_{n\to\infty}\left(|t_n|+|x_n|\right)=\infty,
\end{equation}
then,  for any $\psi\in H^1$,
$$
e^{it_n\Delta}\psi(\cdot+x_n)\rightharpoonup0, \quad in \; H^1.
$$

\item[(ii)]  If $\{z_n\}\subset H^1$ is such that 
$$
z_n\rightharpoonup0 \quad \mbox{and} \quad e^{it_n\Delta}z_n(\cdot+x_n)\rightharpoonup\psi, \quad in \; H^1,
$$
 for some $\psi\in H^1\setminus\{0\}$, then \eqref{pr13} holds.
\end{itemize}
\end{lemma}
\begin{proof}
See Lemma 5.3 in \cite{fxc}.
\end{proof}

Now let us prove that \eqref{pr1} holds at rank $M$. First of all, note that \eqref{pr12} implies that either $\psi^M$ or $\wtpsi^M$ is nontrivial. Let us suppose, without loss of generality, that $\psi^M\neq0$. Assume by induction that \eqref{pr1} holds for $j,k\in\{1,\ldots,M-1\}$. Under the convention $W_n^0=\phi_n$, note that, for all (fixed) $j\in\{1,\ldots,M-1\}$,
\begin{equation}\label{pr14}
e^{it_n^j\Delta}W_n^{j-1}(x+x_n^j)-e^{it_n^j\Delta}W_n^{M-1}(x+x_n^j)-\psi^j= \sum_{k=j+1}^{M-1}e^{i(t_n^j-t_n^k)\Delta}\psi^k(x+x_n^j-x_n^k).
\end{equation}
Part (i) of Lemma \ref{cazlem} implies that the right-hand side of \eqref{pr14} converges weakly to 0 in $H^1$. Thus, since $e^{it_n^j\Delta}W_n^{j-1}(x+x_n^j)\rightharpoonup0$, in $H^1$, it follows from \eqref{pr14} that
$$
z_n=e^{it_n^j}W_n^{M-1}(\cdot+x_n^j)\rightharpoonup0.
$$
 In addition,
$$
e^{i(t_n^M-t_n^j)\Delta}z_n(\cdot+x_n^M-x_n^j)=e^{it_n^M\Delta}W_n^{M-1}(\cdot+x_n^M)\rightharpoonup\psi^M, \quad \mbox{in} \; H^1.
$$
Since $\psi^M\neq0$, an application of Lemma \ref{cazlem} (ii) gives
$$
\lim_{n\to\infty}(|t_n^M-t_n^j|+|x_n^M-x_n^j|)=\infty.
$$
This proves \eqref{pr1} at rank $M$.

Finally, let us prove \eqref{pr2}. If $A_{M_0}=0$ for some $M_0\geq1$, according to the construction, there is nothing to prove. Assume that $A_M>0$ for all $M\geq1$.
Let $(q,r)$ be any $\Hmeio$ admissible pair. The inequality \eqref{pr4.1} and Lemma \ref{strilem} (i) imply that
$$
\|e^{it\Delta}W_n^M\|_{L^q_tL^r_x}\leq\|W_n^M\|^{\theta}_{\Hmeio}\|e^{it\Delta}W_n^M\|^{1-\theta}_{L^\infty_tL^3_x}.
$$
A similar estimate holds with $\wtw_n^M$ instead of $W_n^M$. In view of \eqref{pr3} and \eqref{pr4}, the sequences $\{W_n^M\}$ and $\{\wtw_n^M\}$ are bounded in $\Hmeio$ and $c_M\leq c_1$. Thus, in order to establish \eqref{pr2} it suffices to show that $\lim_{M\to\infty}A_M=0$. But, from \eqref{pr12}, \eqref{pr3}, and  \eqref{pr4},
\begin{equation*}
\begin{split}
\sum_{M=1}^\infty A_M^{10}&\leq cK^2c_1^8 \sum_{M=1}^\infty\left(\|\psi^M\|^2_{\Hmeio}+\|\wtpsi^M\|^2_{\Hmeio}\right)\\
&\leq cK^2c_1^8\limsup_{n\rightarrow \infty}\left(\|\phi_n\|^2_{\Hmeio}+\|\widetilde{\phi}_n\|^2_{\Hmeio}\right)<\infty.
\end{split}
\end{equation*}
This then leads to $\lim_{M\to\infty}A_M=0$ and concludes the proof of the theorem.
\end{proof}

\begin{lemma} \label{lemma1}
In the situation of Theorem \ref{profdec}, we have
\begin{itemize}
\item[(i)] ${\displaystyle \left\| \sum_{j=1}^Me^{-it_n^j\Delta}\psi^j(\cdot-x_n^j)\right\|^4_{L^4_x}=\sum_{j=1}^M\left\|e^{-it_n^j\Delta}\psi^j\right\|^4_{L^4_x}+o_n(1)};$
\item[(ii)]${\displaystyle \lim_{M\to\infty}\left( \lim_{n\to\infty}\|W_n^M\|_{L^4_x} \right)=0};$
\item[(iii)] ${\displaystyle \|\phi_n\|^4_{L^4_x}=\sum_{j=1}^M\|e^{-it_n^j\Delta}\psi^j\|^4_{L^4_x}+\|W_n^M\|^4_{L^4_x}+o_n(1). }$
\end{itemize}
The same conclusions hold for $\wtphi_n,\wtpsi^j$, and $\wtw_n^M$.
\end{lemma}
\begin{proof}
See proof of Lemma 2.3 in \cite{dhr}.
\end{proof}

\begin{lemma} \label{lemma2}
In the situation of Theorem \ref{profdec}, we have
\begin{equation}\label{pr15}
\left\| \sum_{j=1}^Me^{-it_n^j\Delta}\psi^j(\cdot-x_n^j)e^{-it_n^j\Delta}\wtpsi^j(\cdot-x_n^j)\right\|^2_{L^2_x}=\sum_{j=1}^M\left\|e^{-it_n^j\Delta}\psi^j e^{-it_n^j\Delta}\wtpsi^j\right\|^2_{L^2_x}+o_n(1).
\end{equation}
\end{lemma}
\begin{proof}
The proof is similar to that of part (i) in Lemma \ref{lemma1} (see \cite{dhr}). Indeed, reordering indices, we can obtain $M_0\leq M$ such that 
\begin{itemize}
\item[(i)] For $M_0+1\leq j\leq M$, we have, up to a subsequence (in $n$), $|t_n^j|\to\infty$, as $n\to\infty$.
\item[(ii)] For $1\leq j\leq M_0$, we have, up to a subsequence (in $n$), $|t_n^j|\to\overline{t}^j$, as $n\to\infty$.
\end{itemize}
In the first case, Corollary 2.3.7 in \cite{caz} (see also \eqref{L4}) and H\"older's inequality  imply
\begin{equation}\label{pr16}
\left\|e^{-it_n^j\Delta}\psi^j e^{-it_n^j\Delta}\wtpsi^j\right\|^2_{L^2_x}\leq
\left\|e^{-it_n^j\Delta}\psi^j\right\|^2_{L^4_x} \left\|e^{-it_n^j\Delta}\wtpsi^j\right\|^2_{L^4_x}\to0,
\end{equation}
as $n\to\infty$. So, this part has no contribution to \eqref{pr15}.

In the second case, by multiplying the profiles $\psi^j$ and $\wtpsi^j$ by a phase factor, we can assume that $t_n^j=0$. In addition,  the orthogonality condition \eqref{pr1}, implies that $\lim_{n\to\infty}|x_n^j-x_n^k|\to\infty$, for $1\leq j<k\leq M_0$. Thus the ``cross terms'' in the expansion of the left-hand side of \eqref{pr15} goes to zero by the Dominated Convergence Theorem, giving
$$
\left\| \sum_{j=1}^{M_0}\psi^j(\cdot-x_n^j)\wtpsi^j(\cdot-x_n^j)\right\|^2_{L^2_x}=\sum_{j=1}^{M_0}\left\|\psi^j\wtpsi^j\right\|^2_{L^2_x}+o_n(1).
$$
This combined with \eqref{pr16} completes the proof of the lemma.
\end{proof}

\begin{lemma} \label{lemma3}
In the situation of Theorem \ref{profdec}, we have for any fixed $M\geq1$,
\begin{equation}\label{pr17}
\left\| \phi_n\wtphi_n\right\|^2_{L^2_x}=\sum_{j=1}^M\left\|e^{-it_n^j\Delta}\psi^j e^{-it_n^j\Delta}\wtpsi^j\right\|^2_{L^2_x}+ \|W_n^M\wtw_n^M\|^2_{L^2_x}+o_n(1).
\end{equation}
\end{lemma}
\begin{proof}
By using the inequality $||z_1+z_2|^2-|z_1|^2-|z_2|^2|\leq c|z_1||z_2|$, for some constant $c>0$, we deduce
\[
\begin{split}
\Big|\|\phi_n\wtphi_n&-W_n^{M_1}\wtw_n^{M_1}\|^2_{L^2_x}-\|\phi_n\wtphi_n\|^2_{L^2_x}\Big| \\
&\leq \int\left| |\phi_n\wtphi_n-W_n^{M_1}\wtw_n^{M_1}|^2- |\phi_n\wtphi_n|^2- |W_n^{M_1}\wtw_n^{M_1}|^2\right|+|W_n^{M_1}\wtw_n^{M_1}|^2\\
&\lesssim \int |\phi_n\wtphi_n||W_n^{M_1}\wtw_n^{M_1}|+\int|W_n^{M_1}\wtw_n^{M_1}|^2\\
&\lesssim \|\phi_n\wtphi_n\|_{L^2_x}\|W_n^{M_1}\wtw_n^{M_1}\|_{L^2_x}+\|W_n^{M_1}\wtw_n^{M_1}\|^2_{L^2_x}\\
&\lesssim \left(\|\phi_n\|_{L^4_x}\|\wtphi_n\|_{L^4_x}+\|W_n^{M_1}\|_{L^4_x}\|\wtw_n^{M_1}\|_{L^4_x}\right)\|W_n^{M_1}\wtw_n^{M_1}\|_{L^2_x}.
\end{split}
\]
Lemma \ref{lemma1} (ii) implies that the sequences $\{W_n^{M_1}\}$  and $\{\wtw_n^{M_1}\}$ are uniformly bounded in $L^4$. Also, H\"older's inequality combined with Lemma \ref{lemma1} (ii) yield
$$\lim_{M_1\to\infty}\left( \lim_{n\to\infty}\|W_n^{M_1}\wtw_n^{M_1}\|_{L^2_x} \right)=0.
 $$
Thus given any $\varepsilon>0$, we choose $M_1$ and $n$ sufficiently large so that
\begin{equation}\label{pr18}
\Big|\|\phi_n\wtphi_n-W_n^{M_1}\wtw_n^{M_1}\|^2_{L^2_x}-\|\phi_n\wtphi_n\|^2_{L^2_x}\Big|<\varepsilon.
\end{equation}
In a similar fashion, for $M_1$ and $n$ large enough,
\begin{equation}\label{pr19}
\Big|\|W_n^{M}\wtw_n^{M}-W_n^{M_1}\wtw_n^{M_1}\|^2_{L^2_x}-\|W_n^{M}\wtw_n^{1}\|^2_{L^2_x}\Big|<\varepsilon.
\end{equation}
Now note that by definition,
\[
\begin{split}
W_n^{M}\wtw_n^{M}&-W_n^{M_1}\wtw_n^{M_1}=\sum_{j=M+1}^{M_1}e^{-it_n^j\Delta}\psi^j(\cdot-x_n^j) e^{-it_n^j\Delta}\wtpsi^j(\cdot-x_n^j)\\
&+\sum_{j,k=M+1, k\neq j}^{M_1}e^{-it_n^j\Delta}\psi^j(\cdot-x_n^j) e^{-it_n^j\Delta}\wtpsi^j(\cdot-x_n^j)+W_n^{M_1}\sum_{j=M+1}^{M_1}e^{-it_n^j\Delta}\wtpsi^j(\cdot-x_n^j)\\
&+\wtw_n^{M_1}\sum_{j=M+1}^{M_1}e^{-it_n^j\Delta}\psi^j(\cdot-x_n^j) -\sum_{j=M+1}^{M_1}e^{-it_n^j\Delta}\psi^j(\cdot-x_n^j)\sum_{j=1}^{M}e^{-it_n^j\Delta}\wtpsi^j(\cdot-x_n^j)\\
&-\sum_{j=M+1}^{M_1}e^{-it_n^j\Delta}\wtpsi^j(\cdot-x_n^j)\sum_{j=1}^{M}e^{-it_n^j\Delta}\psi^j(\cdot-x_n^j).
\end{split}
\]
Therefore, as in the proof of Lemma \ref{lemma2}, we obtain
\begin{equation}\label{pr20}
W_n^{M}\wtw_n^{M}-W_n^{M_1}\wtw_n^{M_1}=\sum_{j=M+1}^{M_1}e^{-it_n^j\Delta}\psi^j(\cdot-x_n^j) e^{-it_n^j\Delta}\wtpsi^j(\cdot-x_n^j)+o_{n,M_1}(1).
\end{equation}
Consequently, Lemma \ref{lemma2} yields, for $M_1$ and $n$ large enough,
\begin{equation}\label{pr21}
\Big|\|W_n^{M}\wtw_n^{M}-W_n^{M_1}\wtw_n^{M_1}\|^2_{L^2_x}-\sum_{j=M+1}^{M_1}\|e^{-it_n^j\Delta}\psi^j e^{-it_n^j\Delta}\wtpsi^j\|^2_{L^2_x}\Big|<\varepsilon.
\end{equation}
As in \eqref{pr20}, it is not difficult to see that
$$
\phi_n^{M}\wtphi_n^{M}-W_n^{M_1}\wtw_n^{M_1}=\sum_{j=1}^{M_1}e^{-it_n^j\Delta}\psi^j(\cdot-x_n^j) e^{-it_n^j\Delta}\wtpsi^j(\cdot-x_n^j)+o_{n,M_1}(1).
$$
Thus, another application of Lemma \ref{lemma2}, gives
\begin{equation}\label{pr22}
\Big|\|\phi_n^{M}\wtphi_n^{M}-W_n^{M_1}\wtw_n^{M_1}\|^2_{L^2_x}-\sum_{j=1}^{M_1}\|e^{-it_n^j\Delta}\psi^j e^{-it_n^j\Delta}\wtpsi^j\|^2_{L^2_x}\Big|<\varepsilon.
\end{equation}
By combining \eqref{pr18}, \eqref{pr19}, \eqref{pr21}, and \eqref{pr22} we deduce \eqref{pr17}.
\end{proof}

We finish this section collecting all the above estimates to obtain the following energy decomposition.

\begin{proposition}[Energy Decomposition]
In the same situation of Theorem \ref{profdec}, for each $M\geq1$,
\begin{equation}\label{ED}
E(\phi_n,\wtphi_n)=\sum_{j=1}^ME(e^{-it_n^j\Delta}\psi^j(\cdot-x_n^j), e^{-it_n^j\Delta}\wtpsi^j(\cdot-x_n^j))+E(W_n^M,\wtw_n^M)+o_n(1).
\end{equation}
\end{proposition}
\begin{proof}
Combine Lemma \ref{lemma2} (iii) with Lemma \ref{lemma3}.
\end{proof}

\section{Construction of a critical solution}\label{ConstCri}

In this section we will construct a critical solution. We follow the exposition in Holmer-Roudenko \cite{hr} (see also  Duyckaerts-Holmer-Roudenko \cite{dhr} and Fang-Xie-Cazenave \cite{fxc}), which was based in the ideas introduced by Kenig-Merle \cite{km}. We start with the following definition.
\begin{definition}
Suppose $(u_0, v_0)\in H^1\times H^1$ and let $(u,v)$ be the corresponding $H^1$-solution to the three-dimensional cubic nonlinear Schr\"odinger system \eqref{3nls}. Let $[0,T^{\ast})$ be the maximal forward interval of existence. We shall say that SC($u_0, v_0$) holds if 
$$
T^{\ast}=\infty \quad \textrm{and} \quad \|(u,v)\|_{S(\Hmeio)\times S(\Hmeio)}<\infty.
$$
\end{definition}

Our goal is to show that SC($u_0, v_0$) holds for any $(u_0, v_0)\in \mathcal{K}$. We already know from Theorem \ref{globalPastor} that, for any $(u_0, v_0)\in \mathcal{K}$,  $T^{\ast}=\infty$ and the solution $(u,v)$ is uniformly bounded in $H^1\times H^1$ (see Remark \ref{Rema}). Moreover, as we already observed, if $\|u\|_{S(\Hmeio)}+\|v\|_{S(\Hmeio)}<\infty$ then Proposition \ref{h1scat} implies that $(u,v)$ scatters. Let
\begin{equation*}
\mathcal{L}:=\left\{(u_0,v_0)\in \mathcal{K}: \|(u,v)\|_{S(\Hmeio)\times S(\Hmeio)}<\infty \right\}.
\end{equation*}
With these remarks, we want to show that 
\begin{equation*}
\mathcal{L}= \mathcal{K}.
\end{equation*}

Let $M_c$ be the number defined as the supremum over all $\delta$ for which the following statement holds:\\

\textit{``If $(u_0, v_0)\in H^1\times  H^1$ satisfies \eqref{EM1}  and $M(u_0,v_0)E(u_0,v_0)<\delta$ then SC($u_0, v_0$) holds.''}\\

Note that there always exists a $\delta>0$ such that the above statement is true. Indeed, if $M(u_0,v_0)E(u_0,v_0)<\delta$, by interpolation and Lemma \ref{grolemma} (i), we have 
\begin{equation*}
\begin{split}
\|u_0\|^4_{\dot{H}^{1/2}}+\|v_0\|^4_{\dot{H}^{1/2}}&\leq \|\nabla u_0\|_{L^2}^{2}\|u_0\|_{L^2}^{2} + \|\nabla v_0\|_{L^2}^{2}\|v_0\|_{L^2}^{2} \\
&< M(u_0,v_0)A(u_0,v_0)\\
&\leq 6 M(u_0,v_0)E(u_0,v_0)\\
&< 6 \delta.
\end{split}
\end{equation*}
Lemma \ref{strilem} (i) now implies that $\|(e^{it\Delta}u_0, e^{it\Delta}v_0)\|_{S(\Hmeio)\times S(\Hmeio)}$ is sufficiently small.
Therefore, the small data global theory (Theorem \ref{teocri}) asserts that for a small enough $\delta>0$ we have a global solution that scatters.

Our purpose here is then to show that $M_c\geq M(P,Q)E(P,Q)$. Assume, by contradiction, that $M_c < M(P,Q)E(P,Q)$. Therefore, there exists a sequence of solutions $(u_n,v_n)$ of system \eqref{3nls} with corresponding initial data $(u_{n,0}, v_{n,0})$ such that $(u_{n,0}, v_{n,0})\in \mathcal{K}$ and
$$
M(u_n,v_n)E(u_n,v_n)  \searrow M_c, \quad \textrm{as} \quad n\rightarrow \infty,
$$
for which SC($u_{n,0}, v_{n,0}$) does not hold for any $n\in \mathbb{N}$. As a consequence, we have 
\begin{equation}\label{SCunvn}
\|(u_n,v_n)\|_{S(\Hmeio)\times S(\Hmeio)}=\infty, \quad \textrm{for any} \quad n\in \mathbb{N}.
\end{equation}

Our main goal in this section is to apply the profile decomposition result stated in Theorem \ref{profdec} to construct a {\textit{critical solution} that does not scatters. Indeed, by the scaling symmetry \eqref{scalsym} (see also \eqref{scalL2}) we can re-scale the sequence described before such that $M(u_{n,0},v_{n,0})=1$, for any $n\in \mathbb{N}$. In this case we have 
\begin{equation*}
A(u_{n,0},v_{n,0})<M(P,Q)A(P,Q)
\end{equation*}
and
\begin{equation}\label{UN1}
E(u_n,v_n)  \searrow M_c, \quad \textrm{as} \quad n\rightarrow \infty.
\end{equation}

In view of the assumption $M_c < M(P,Q)E(P,Q)$, without loss of generality, there exists $\overline{\delta}\in (0,1)$ such that, for every $n\in \N$,
$$
E(u_n,v_n) \leq \overline{\delta}M(P,Q)E(P,Q).
$$
Since $(u_{n,0}, v_{n,0})\in \mathcal{K}$, relation \eqref{EM1} is satisfied, for any $n\in \N$, which implies by  Lemma \ref{grolemma} (ii) that
\begin{equation}\label{UN0}
A(u_{n,0},v_{n,0})\leq \overline{\delta} M(P,Q)A(P,Q).
\end{equation}
Therefore, $\{u_{n,0}\}_{n\in \mathbb{N}}$ and $\{v_{n,0}\}_{n\in \mathbb{N}}$ are two uniformly bounded sequences in $H^1$ and we can apply the profile expansion (Theorem \ref{profdec}) to obtain
\begin{equation}\label{(un0vn0)}
(u_{n,0}, v_{n,0})=\sum_{j=1}^M(e^{-it_n^j\Delta}\psi^j(\cdot-x_n^j),e^{-it_n^j\Delta}\wtpsi^j(\cdot-x_n^j))+(W_n^M, \wtw_n^M).
\end{equation}

By the Pythagorean asymptotic expansions \eqref{pr3} and \eqref{pr4}, we have for all $M\in \mathbb{N}$,
\begin{equation}\label{Psi}
\sum_{j=1}^{M} M(\psi^j, \wtpsi^j) \leq \lim_{n\rightarrow \infty}M(u_{n,0},v_{n,0})=1.
\end{equation}
Moreover, another application of \eqref{pr3} and \eqref{pr4}, with $s=1$, and \eqref{UN0}, yield
\begin{equation}\label{Apsi}
\sum_{j=1}^{M} A(\psi^j, \wtpsi^j) \leq \limsup_{n\rightarrow \infty}A(u_{n,0},v_{n,0})\leq  \overline{\delta} M(P,Q)A(P,Q).
\end{equation}
Therefore, for all $j,n\in \N$, we deduce from \eqref{Psi} and \eqref{Apsi},
$$
M\left(e^{-it_n^j\Delta}\psi^j(\cdot-x_n^j), e^{-it_n^j\Delta}\wtpsi^j(\cdot-x_n^j)\right) A\left(e^{-it_n^j\Delta}\psi^j(\cdot-x_n^j), e^{-it_n^j\Delta}\wtpsi^j(\cdot-x_n^j)\right)
$$
\begin{equation*}
=M\left(\psi^j, \wtpsi^j\right) A\left(\psi^j, \wtpsi^j\right) \leq  \overline{\delta} M(P,Q)A(P,Q).
\end{equation*}
From Lemma \ref{grolemma} (i), we obtain
$$
E(e^{-it_n^j\Delta}\psi^j(\cdot-x_n^j), e^{-it_n^j\Delta}\wtpsi^j(\cdot-x_n^j)) \geq \frac{1}{6} A(\psi^j, \wtpsi^j) \geq 0.
$$
A completely similar analysis yields, for all $M\in \mathbb{N}$, 
$$\limsup_{n\rightarrow \infty}M(W_n^M, \wtw_n^M)\leq 1,$$ 
$$\limsup_{n\rightarrow \infty}A(W_n^M, \wtw_n^M) \leq  \overline{\delta} M(P,Q)A(P,Q),$$ 
and, for $n$ sufficiently large (depending on $M$),
$$E(W_n^M, \wtw_n^M)\geq 0.$$

The energy decomposition \eqref{ED} and \eqref{UN1} allow us to conclude 
$$
\sum_{j=1}^M\limsup_{n\rightarrow \infty} E(e^{-it_n^j\Delta}\psi^j(\cdot-x_n^j), e^{-it_n^j\Delta}\wtpsi^j(\cdot-x_n^j))+\limsup_{n\rightarrow \infty} E(W_n^M,\wtw_n^M)
$$
$$
=E(u_n,v_n)  \searrow M_c
$$
and, since the energy of all terms involved is positive, we have for all $j\geq 1$,
\begin{equation}\label{EnergyPsi}
\limsup_{n\rightarrow \infty} E(e^{-it_n^j\Delta}\psi^j(\cdot-x_n^j), e^{-it_n^j\Delta}\wtpsi^j(\cdot-x_n^j)) \leq M_c.
\end{equation}

The next lemma is the heart of the analysis. It asserts that the sequence 
$$
\{(\psi^j, \wtpsi^j)\}_{j\in \N}\subset H^1 \times H^1
$$ 
obtained in the above discussion has at most one nonzero element.

\begin{lemma}\label{Uniq}
There exists at most one $j\in \N$ such that $(\psi^j, \wtpsi^j)\neq (0,0)$.
\end{lemma}

\begin{proof}
Assume that $(\psi^j, \wtpsi^j)\neq (0,0)$ for  $j\in J$, where $J\subset \N$ has more than one element. In view of \eqref{Psi}, for any $j\in J$,
\begin{equation}\label{psij}
M(\psi^j, \wtpsi^j)< 1.
\end{equation}
Therefore, by \eqref{Apsi} and \eqref{EnergyPsi}, we have 
\begin{equation}\label{MA}
M(\psi^j, \wtpsi^j)A(\psi^j, \wtpsi^j)<  \overline{\delta} M(P,Q)A(P,Q),
\end{equation}
and
\begin{equation}\label{ME}
M(\psi^j, \wtpsi^j)E(e^{-it_n^j\Delta}\psi^j(\cdot-x_n^j), e^{-it_n^j\Delta}\wtpsi^j(\cdot-x_n^j))  < M_c,
\end{equation}
for any $j\in J$ and $n$ sufficiently large.

Next, we consider the sequence $\{t^j_n\}_{n\in \mathbb{N}}$, for a given $j\in J$. If this sequence is bounded, passing to a subsequence if necessary, we can assume $t^j_n\rightarrow \bar{t}^j \in \R$. Define 
$$
(\eta^j,\widetilde{\eta}^j)=\textrm{NLSs}(\bar{t}^j)\left(e^{-i\bar{t}^j\Delta}\psi^j, e^{-i\bar{t}^j\Delta}\wtpsi^j\right),
$$
where $\{\textrm{NLSs}(t)\}_{t\in \R}$ denotes the flow of the nonlinear system \eqref{3nls}. Since $\mathcal{K}$ is a closed subset of $H^1\times H^1$, by \eqref{MA} and \eqref{ME}, we have $(e^{-i\bar{t}^j\Delta}\psi^j, e^{-i\bar{t}^j\Delta}\wtpsi^j) \in \mathcal{K}$. Therefore $\textrm{NLSs}(t)\left(e^{-i\bar{t}^j\Delta}\psi^j, e^{-i\bar{t}^j\Delta}\wtpsi^j\right)$ is a global solution and $(\eta^j,\widetilde{\eta}^j)$ satisfies conditions \eqref{EM} and \eqref{EM1} in view of Theorem \ref{globalPastor}. By the continuity of the linear and nonlinear flows, we also deduce
\begin{equation}\label{I1}
\|\textrm{NLSs}(-t^j_n)(\eta^j,\widetilde{\eta}^j)-(e^{-i{t^j_n}\Delta}\psi^j, e^{-it^j_n\Delta}\wtpsi^j)\|_{H^1\times H^1}  \rightarrow 0, \quad \textrm{as} \quad n\rightarrow \infty.
\end{equation}

On the other hand, if $\{t^j_n\}_{n\in \mathbb{N}}$ is not bounded, we may assume $|t^j_n|\rightarrow \infty$. Passing to a subsequence if necessary, we have $t^j_n\rightarrow +\infty$ or $t^j_n\rightarrow -\infty$. We first claim that 
\begin{equation}\label{L4}
\lim_{n\rightarrow \infty}\Big\{\|e^{-i{t^j_n}\Delta}\psi^j\|_{L^{4}_x}+\|e^{-it^j_n\Delta}\wtpsi^j\|_{L^{4}_x}\Big\}=0.
\end{equation}
Indeed, for $\phi \in L^{4/3}_x\cap \dot{H}^{3/4}$ we have by Sobolev embedding and Lemma \ref{Decay},
\begin{equation*}
\begin{split}
\|e^{-i{t^j_n}\Delta}\psi^j\|_{L^{4}_x}&=\|e^{-i{t^j_n}\Delta}(\psi^j-\phi)\|_{L^{4}_x}+\|e^{-i{t^j_n}\Delta}\phi^j\|_{L^{4}_x}\\
&\leq c \|\psi^j-\phi\|_{\dot{H}^{3/4}} + c|t^j_n|^{-3/4}\|\phi\|_{L^{4/3}_x}.
\end{split}
\end{equation*}
The same discussion also holds for $\wtpsi^j$ and thus a density argument shows the desired limit and establishes the claim.

In view of \eqref{L4}, the definition of the energy \eqref{econ}, H\"older's inequality and \eqref{EnergyPsi}, we deduce
$$\frac{1}{2}A(\psi^j, \wtpsi^j)=\limsup_{n\rightarrow \infty} E(e^{-it^j_n\Delta}\psi^j(\cdot-x^j_n), e^{-it^j_n\Delta}\wtpsi^j(\cdot-x^j_n))  \leq M_c.$$
Therefore, by the existence of wave operator (Theorem \ref{waveop}), there exists  $(\eta^j,\widetilde{\eta}^j)\in H^1\times H^1$ such that \eqref{I1} also holds. Moreover, by \eqref{waveop1}, \eqref{waveop2} and \eqref{psij}
\begin{equation}\label{I2}
M(\eta^j,\widetilde{\eta}^j)=M(\psi^j, \wtpsi^j)<1,
\end{equation}
\begin{equation}\label{I3}
M(\eta^j,\widetilde{\eta}^j)A(\eta^j,\widetilde{\eta}^j)<M(P,Q)A(P,Q),
\end{equation}
\begin{equation}\label{I4}
E(\eta^j,\widetilde{\eta}^j)=\frac{1}{2}A(\psi^j, \wtpsi^j)\leq M_c.
\end{equation}
Note that here, the constructed pair $(\eta^j,\widetilde{\eta}^j)$ also satisfies conditions \eqref{EM} and \eqref{EM1} in view of \eqref{I2}-\eqref{I4}. 

Next we observe that \eqref{I1} and \eqref{I4} imply
\begin{equation*}
M(\eta^j,\widetilde{\eta}^j) E(\eta^j,\widetilde{\eta}^j)< M_c,
\end{equation*}
and thus, for any $j\in J$, the definition of $M_c$ and \eqref{I3} yield
\begin{equation}\label{NLSs}
\|\textrm{NLSs}(t)(\eta^j,\widetilde{\eta}^j))\|_{S(\Hmeio)\times S(\Hmeio)}<\infty
\end{equation}
and
\begin{equation}\label{NLSs2}
\sup_{t\in \R}\|\textrm{NLSs}(t)(\eta^j,\widetilde{\eta}^j))\|_{H^1\times H^1}<\infty,
\end{equation}
where in the last inequality we have used Remark \ref{Rema}.

Now, setting 
\begin{equation}\label{Decomp}
(u_{n,0}, v_{n,0})=\sum_{j=1}^M\textrm{NLSs}(-t^j_n)(\eta^j,\widetilde{\eta}^j)(\cdot-x_n^j)+(K^M_n, \widetilde{K}^M_n),
\end{equation}
we have by \eqref{(un0vn0)},
\begin{equation*}
\begin{split}
(K^M_n, \widetilde{K}^M_n)=&\sum_{j=1}^M(e^{-i{t^j_n}\Delta}\psi^j(\cdot-x^j_n), e^{-it^j_n\Delta}\wtpsi^j(\cdot-x^j_n))-\textrm{NLSs}(-t^j_n)(\eta^j,\widetilde{\eta}^j)(\cdot-x^j_n)\\
&+(W^M_n, \wtw_n^M)
\end{split}
\end{equation*}
which implies, in view of \eqref{I1} and \eqref{pr2},
\begin{equation}\label{Kmn1}
\lim_{M\to\infty}\Big[\lim_{n\to\infty}\|(e^{it\Delta}K_n^M, e^{it\Delta}\widetilde{K}_n^M)\|_{S(\Hmeio)\times S(\Hmeio)}\Big]=0.
\end{equation}

Let
$$
(u_n(t), v_{n}(t))=\textrm{NLSs}(t)(u_{n,0}, v_{n,0}).
$$
Our aim is to show that 
$$
(u_n(t), v_{n}(t)) \approx \sum_{j=1}^M\textrm{NLSs}(t-t^j_n)(\eta^j,\widetilde{\eta}^j)(\cdot-x_n^j)
$$
in the spirit of the long time perturbation theory (Proposition \ref{LTstabteo}). To this end we define
\begin{equation}\label{wnjznj}
(w_n^j(t), z_n^j(t))=\textrm{NLSs}(t-t^j_n)(\eta^j,\widetilde{\eta}^j)(\cdot-x_n^j)
\end{equation}
and
\begin{equation}\label{uMnvMn}
(u^M_n(t), v^M_{n}(t))=\sum_{j=1}^M(w_n^j(t), z_n^j(t)).
\end{equation}
We can easily check that $(u^M_n, v^M_{n})$ satisfies the following system
\begin{equation}\label{3nlstil2}
\begin{cases}
i\partial_t{u^M_n}+\Delta {u^M_n}+(|{u^M_n}|^2+\beta |{v^M_n}|^2){u^M_n}=e^M_{n,1},\\
i\partial_t{v^M_n}+\Delta {v^M_n}+(|{v^M_n}|^2+\beta |{u^M_n}|^2){v^M_n}=e^M_{n,2},
\end{cases}
\end{equation}
where 
$$
e^M_{n,1}=(|{u^M_n}|^2+\beta |{v^M_n}|^2){u^M_n}-\sum_{j=1}^M(|{w^j_n}|^2+\beta |{z^j_n}|^2){w^j_n}
$$
and 
$$
e^M_{n,2}=(|{v^M_n}|^2+\beta |{u^M_n}|^2){v^M_n}-\sum_{j=1}^M(|{z^j_n}|^2+\beta |{w^j_n}|^2){z^j_n}.
$$
In view of \eqref{Decomp}, we first note that $(u_n(0), v_{n}(0))-(u^M_n(0), v^M_{n}(0))=(K^M_n, \widetilde{K}^M_n)$, therefore, by \eqref{Kmn1}, given $\varepsilon>0$ there exist $n_0(\varepsilon), M_0(\varepsilon)\in \N$ sufficiently large such that
\begin{equation}\label{CLAIM1}
\|(e^{it\Delta}K_n^{M_0}, e^{it\Delta}\widetilde{K}_n^{M_0})\|_{S(\Hmeio)\times S(\Hmeio)}<\varepsilon,
\end{equation}
for all $n\geq n_0(\varepsilon)$.\\

\noindent {\bf Claim 1.} For a fixed $M\in \N$ there exists $n_1(M)\in\N$ such that, for $n\geq n_1(M)$,
\begin{equation}\label{CLAIM3}
\|(e^{M}_{n,1}, e^{M}_{n,2})\|_{S'(\dot{H}^{-1/2})\times S'(\dot{H}^{-1/2})}\leq \varepsilon(A),
\end{equation}
where $\varepsilon(A)>0$ is given by the long time perturbation theory (Proposition \ref{LTstabteo}).\\
Indeed, by the elementary inequality
$$
\Big||\sum_{j=1}^Ma_j|^2\sum_{j=1}^Mb_j-\sum_{j=1}^M|a_j|^2b_j\Big|\leq C_M\sum_{
( j,k,l)\in L} |c_j||c_k||c_l|,
$$
where $c_i=a_i \textrm{ or } b_i$ and $L=\{(j,k,l)\in\N^3: 1\leq j,l,k\leq M,\; j\neq l \textrm{ or } j\neq k \textrm{ or } l\neq k\}$, we deduce the following estimate
$$
\|(e^{M}_{n,1}, e^{M}_{n,2})\|_{S'(\dot{H}^{-1/2})\times S'(\dot{H}^{-1/2})}\leq cC_M \sum_{
( j,l,k)\in L}\||f_n^j||f_n^l||f_n^k|\|_{S'(\dot{H}^{-1/2})},
$$
where $f_n^i=w^i_n \textrm{ or } z^i_n$ (recall definition \eqref{wnjznj}). Note that $f_n^i\in S(\Hmeio)$ by \eqref{NLSs} and therefore the right hand side of the last inequality is finite. Furthermore, assume, without loss of generality, that $j\neq k$. By a density argument and \eqref{pr1} it is easy to see that 
$$
\||f_n^j||f_n^k||f_n^l|\|_{S'(\dot{H}^{-1/2})}\rightarrow 0, \quad \textrm{as} \quad n\rightarrow \infty,
$$
which implies \eqref{CLAIM3}.\\

\noindent {\bf Claim 2.} There exists $A>0$ such that, for any $M\in \N$, there exists $n_2(M)\in \N$ such that
\begin{equation}\label{CLAIM2}
\|(u^M_n, v^M_{n})\|_{S(\Hmeio)\times S(\Hmeio)}<A, \quad \textrm{for all} \quad n>n_2(M).
\end{equation}
Indeed, we already know (see \eqref{Psi} and \eqref{Apsi}) that there exists $C_1>0$ such that
$$\sum_{j=1}^{\infty}\|(\psi^j, \wtpsi^j) \|^2_{H^1\times H^1}\leq C_1.$$
Therefore, we can choose $M_1\in \mathbb{N}$ with
$$\sum_{j=M_1}^{\infty}\|(\psi^j, \wtpsi^j) \|^2_{H^1\times H^1}\leq \delta/2,$$
where $\delta>0$ is a sufficiently small number to be chosen later.
Moreover, by \eqref{I1}), for all $M> M_1$ there exists $n_3(M)\geq 1$ such that for all $n\geq n_3(M)$, we have
\begin{equation}\label{SUMNLSs}
\sum_{j=M_1}^{M}\|\textrm{NLSs}(-t^j_n)(\eta^j,\widetilde{\eta}^j)\|^2_{H^1\times H^1}\leq \delta.
\end{equation}
Therefore, if $\delta<\delta_{sd}$, where $\delta_{sd}>0$ is given by Corollary \ref{teocri2}, the small data global theory in $H^1\times H^1$ yields, for all $n\geq n_3(M)$,
\begin{equation*}
\sum_{j=M_1}^{M}\|(w_n^j(t), z_n^j(t))\|^2_{H^1\times H^1}\leq C_{sd} \sum_{j=M_1}^{M}\|\textrm{NLSs}(-t^j_n)(\eta^j,\widetilde{\eta}^j)\|^2_{H^1\times H^1}\leq C_{sd}\delta_{sd}.
\end{equation*}
On the other hand, by the definition of $(u^M_n, v^M_{n})$ (see \eqref{uMnvMn}),
\begin{equation*}
\begin{split}
\Big\|(u^M_n(t), v^M_{n}(t))\Big\|^2_{H^1\times H^1}=&\sum_{j=1}^{M_1-1}\|(w_n^j(t), z_n^j(t))\|^2_{H^1\times H^1}+\sum_{j=M_1}^{M}\|(w_n^j(t), z_n^j(t))\|^2_{H^1\times H^1}\\
&+2\sum_{M_1\leq l\neq k\leq M}\big((w_n^l(t), z_n^l(t)), (w_n^k(t), z_n^k(t))\big)_{H^1\times H^1}.
\end{split}
\end{equation*}
For every $l\neq k$ we deduce from \eqref{pr1} that 
$$
\sup_{t\in \R}\big|\big((w_n^l(t), z_n^l(t)), (w_n^k(t), z_n^k(t))\big)_{H^1\times H^1}\big|\rightarrow 0, \quad \textrm{as} \quad n\rightarrow \infty.
$$
Thus, from the last three relations and \eqref{NLSs2}, for all $M> M_1$, there exist $A_1>0$ (independent of $M$) and $n_4(M)\geq 1$ such that
\begin{equation}\label{A_1}
\sup_{t\in \R}\Big\|(u^M_n(t), v^M_{n}(t))\Big\|_{H^1\times H^1}\leq A_1,
\end{equation}
for all $n\geq n_4(M)$.

Next, we make a similar estimate for the norm $L^{5}_{x,t}$ (recall that $(5,5)$ is an $\Hmeio$ admissible pair). It follows from the elementary inequality
$$
\Big||\sum_{j=1}^Ma_j|^{\alpha}-\sum_{j=1}^M|a_j|^{\alpha}\Big|\leq C_{\alpha,M}\sum_{l\neq k} |a_l||a_k|^{\alpha-1},
$$
for every $\alpha>1$, $M\geq 1$, and $a_j\in \C$, that 
\begin{equation*}
\begin{split}
\Big\|u^M_n\Big\|^{5}_{L^{5}_{x,t}}\leq \sum_{j=1}^{M_1-1}\|w_n^j\|^{5}_{L^{5}_{x,t}}+\sum_{j=M_1}^{M}\|w_n^j\|^{5}_{L^{5}_{x,t}}+C_{M}\!\!\!\sum_{M_1\leq l\neq k\leq M}\int_{\R^4} |w_n^l||w_n^k||w_n^k|^{3} dx dt
\end{split}
\end{equation*}
for all $M>M_1$. From \eqref{NLSs}, there exists $A(M_1)>0$ such that, for all $n\in \N$,
$$
\sum_{j=1}^{M_1-1}\|w_n^j\|^{5}_{L^{5}_{x,t}}\leq A(M_1).
$$
On the other hand, if $\delta>0$ is small enough, we deduce from the small data global theory in $\Hmeio\times \Hmeio$ (Theorem \ref{teocri}), Strichartz estimate (Lemma \ref{strilem} (i)), \eqref{wnjznj} and \eqref{SUMNLSs} that
\begin{equation*}
\begin{split}
\sum_{j=M_1}^{M}\|w_n^j\|^{5}_{L^{5}_{x,t}}&\leq 2^5 \sum_{j=M_1}^{M}\|\textrm{NLSs}(-t^j_n)(\eta^j,\widetilde{\eta}^j)\|^5_{H^1\times H^1}\\
&\leq 2^5\sum_{j=M_1}^{M}\|\textrm{NLSs}(-t^j_n)(\eta^j,\widetilde{\eta}^j)\|^2_{H^1\times H^1}\leq 2^6\delta,
\end{split}
\end{equation*}
for all $n\geq n_3(M)$.
Moreover, if $M_1\leq l\neq k\leq M$, H\"older's inequality and another application of Theorem \ref{teocri} yield
\begin{equation*}
\begin{split}
\int_{\R^4} |w_n^l||w_n^k||w_n^k|^{3} dx dt&\leq \|w_n^j\|^{3}_{L^{5}_{x,t}}\Big(\int_{\R^4} |w_n^l|^{5/2}|w_n^k|^{5/2}\Big)^{2/5}\\
&\leq 2^3\|\textrm{NLSs}(-t^j_n)(\eta^j,\widetilde{\eta}^j)\|^3_{H^1\times H^1}\Big(\int_{\R^4} |w_n^l|^{5/2}|w_n^k|^{5/2}\Big)^{2/5}
\end{split}
\end{equation*}
and we deduce from \eqref{pr1} that the right-hand side of the above inequality goes to $0$ as $n\rightarrow \infty$. Collecting  the last four relations, for all $M> M_1$, there exist $A_2>0$ (independent of $M$) and $n_5(M)\geq 1$ such that, for $n\geq n_5(M)$,
$$
\|u^M_n\|_{L^{5}_{x,t}}\leq A_2.
$$
An analogous estimate can be done for $v^M_n$, which implies 
\begin{equation}\label{A_2}
\|u^M_n\|_{L^{5}_{x,t}}+\|v^M_{n}\|_{L^{5}_{x,t}}\leq A_2,
\end{equation}
for $n\geq n_5(M)$.

Next, taking $n_6(M)= \max\{n_4(M), n_5(M)\}$, using \eqref{A_1} and \eqref{A_2} an interpolation argument and the fact that $H^1(\R^3)\subset L^3(\R^3)$ we have, for all $M>M_1$,
\begin{equation}\label{L8L4}
\begin{split}
\|u^M_n\|_{L^{8}_{t}L^{4}_{x}}\leq \|u^M_n\|^{3/8}_{L^{\infty}_{t}L^{3}_{x}}\|u^M_n\|^{5/8}_{L^{5}_{x,t}} \leq \|u^M_n\|^{3/8}_{L^{\infty}_{t}H^1}\|u^M_n\|^{5/8}_{L^{5}_{x,t}}\leq A_1^{3/8}A_2^{5/8},
\end{split}
\end{equation}
for all $n\geq n_6(M)$. Analogously, we can also obtain the same estimate for $v^M_n$.

Now, since $(u^M_n, v^M_{n})$ satisfies  system \eqref{3nlstil2} we can apply Lemma \ref{strilem} to the integral formulation to deduce (note that $(8/5,4)$ is a $\dot{H}^{-1/2}$ admissible pair)
\begin{equation*}
\begin{split}
\|(u^M_n, v^M_{n})\|_{S(\Hmeio)\times S(\Hmeio)}\leq &\;\; c \|(u^M_n(0), v^M_{n}(0))\|_{H^1\times H^1}\\
&+ c\|u^M_n\|^3_{L^{8}_{t}L^{4}_{x}}+c\|v^M_n\|^3_{L^{8}_{t}L^{4}_{x}}\\
&+c\|(e^{M}_{n,1}, e^{M}_{n,2})\|_{S'(\dot{H}^{-1/2})\times S'(\dot{H}^{-1/2})},
\end{split}
\end{equation*}
which conclude the proof of \eqref{CLAIM2} in view of \eqref{A_1}, \eqref{L8L4} and \eqref{CLAIM3}.\\

Finally, taking $\varepsilon=\varepsilon(A)$ we can find $M_0(\varepsilon(A)), n_0(\varepsilon(A))\in \N$ such that \eqref{CLAIM1} holds. Thus setting $n_3=\max \{n_0(\varepsilon(A)), n_1(M_0), n_2(M_0)\}$ in view of \eqref{CLAIM1}-\eqref{CLAIM3}, we can apply the long time perturbation theory (Proposition \ref{LTstabteo}) to deduce
\begin{equation}\label{SCunvn2}
\|(u_n,v_n)\|_{S(\Hmeio)\times S(\Hmeio)}<\infty,
\end{equation}
for every $n\geq n_3$, a contradiction with \eqref{SCunvn}.
\end{proof}

Now we have all tools to prove the existence of a critical solution.

\begin{theorem}[Existence of a critical solution]\label{ECE} If $M_c < M(P,Q)E(P,Q)$, there exists $(u_{c,0},v_{c,0})\in H^1\times H^1$ such that the corresponding solution $(u_c,v_c)$ of the system \eqref{3nls} is global in $H^1\times H^1$ with
\begin{enumerate}
\item[(i)] $M(u_c,v_c)=1;$
\item[(ii)]$E(u_c,v_c)=M_c<M(P,Q)E(P,Q);$
\item[(iii)] $M(u_{c},v_{c})A(u_{c},v_{c})<M(P,Q)A(P,Q);$
\item[(iv)] $\|u_c\|_{S(\Hmeio)}+\|v_c\|_{S(\Hmeio)}=\infty.$
\end{enumerate}
\end{theorem} 
\begin{proof}
Let $\{(u_{n,0}, v_{n,0})\}_{n\in \N}\subset H^1\times H^1$ be the sequence constructed in the introduction of this section and let $ \{(\psi^j, \wtpsi^j)\}_{j\in \N}\subset H^1\times H^1$ be the sequence constructed in Lemma \ref{Uniq}. In view of Lemma \ref{Uniq} we can assume, without loss of generality, that $(\psi^1, \wtpsi^1)\neq (0,0)$ and $(\psi^j, \wtpsi^j)=(0,0)$, for all $j\geq 2$. For simplicity we omit the index $1$ in what follows. Using the same argument as the one in the previous case, we can find $(\eta,\widetilde{\eta})\in H^1\times H^1$ such that
\begin{equation}\label{III1}
M(\eta,\widetilde{\eta})=M(\psi, \wtpsi)\leq 1,
\end{equation}
\begin{equation}\label{III11}
M(\eta,\widetilde{\eta})A(\eta,\widetilde{\eta})<M(P,Q)A(P,Q),
\end{equation}
\begin{equation}\label{III2}
E(\eta,\widetilde{\eta})=\frac{1}{2}A(\psi, \wtpsi)\leq M_c,
\end{equation}
and
\begin{equation}\label{III3}
\|\textrm{NLSs}(-t_n)(\eta,\widetilde{\eta})-(e^{-i{t_n}\Delta}\psi, e^{-it_n\Delta}\wtpsi)\|_{H^1\times H^1}  \rightarrow 0, \quad \textrm{as} \quad n\rightarrow \infty.
\end{equation}

In view of \eqref{III1}-\eqref{III2} it is easy to see that $(\eta,\widetilde{\eta})$ satisfies conditions \eqref{EM} and \eqref{EM1}. Set
$$
(K^M_n, \widetilde{K}^M_n)=(W^M_n, \wtw_n^M )+(e^{-i{t_n}\Delta}\psi(\cdot-x_n), e^{-it_n\Delta}\wtpsi(\cdot-x_n))- \textrm{NLSs}(-t_n)(\eta,\widetilde{\eta})(\cdot-x_n).
$$
Using the Strichartz estimates (Lemma \ref{strilem}), \eqref{III3}, and \eqref{pr2}, we deduce
\begin{equation}\label{Kmn}
\lim_{M\to\infty}\Big[\lim_{n\to\infty}\big( \|e^{it\Delta}K_n^M\|_{S(\Hmeio)}+\|e^{it\Delta}\widetilde{K}_n^M\|_{S(\Hmeio)}\Big)\Big]=0.
\end{equation}
By definition of $K^M_n$ and $\widetilde{K}^M_n$, we can write (see \eqref{Decomp})
\begin{equation}\label{un0vn0}
(u_{n,0}, v_{n,0})=\textrm{NLSs}(-t_n)(\eta,\widetilde{\eta})(\cdot-x_n)+(K^M_n, \widetilde{K}^M_n).
\end{equation}

Let $(u_c,v_c)$ be the solution of  \eqref{3nls} with initial data $(\eta,\widetilde{\eta})$, that is, 
\begin{equation}\label{ucvc}
(u_c(t),v_c(t))=\textrm{NLSs}(t)(\eta,\widetilde{\eta}).
\end{equation}
Since $(\eta,\widetilde{\eta})$ satisfies conditions \eqref{EM} and \eqref{EM1},  by Theorem \ref{globalPastor}, the solution \eqref{ucvc} exists globally and inequality (iii) in the statement of Theorem \ref{ECE} holds. 

We claim that 
\begin{equation}\label{claim}
\|u_c\|_{S(\Hmeio)}+\|v_c\|_{S(\Hmeio)}=\infty.
\end{equation}

Indeed, suppose that $\|u_c\|_{S(\Hmeio)}+\|v_c\|_{S(\Hmeio)}<\infty$. The idea is to use the long time perturbation theory to obtain a contradiction. If we define
\begin{equation}\label{tuntvn}
(\widetilde{u}_n(t), \widetilde{v}_n(t))=\textrm{NLSs}(t-t_n)(\eta,\widetilde{\eta})(\cdot-x_n),
\end{equation}
then
$$\|\widetilde{u}_n\|_{S(\Hmeio)}+\|\widetilde{u}_n\|_{S(\Hmeio)}=\|u_c\|_{S(\Hmeio)}+\|v_c\|_{S(\Hmeio)}<\infty.$$
Moreover, if $(u_n(t), v_n(t))=\textrm{NLSs}(t)(u_{n,0}, v_{n,0})$, then by \eqref{un0vn0} and \eqref{tuntvn},
\begin{equation*}
\|e^{it\Delta}(u_n(0)-\widetilde{u}_n(0))\|_{S(\Hmeio)}+\|e^{it\Delta}(v_n(0)-\widetilde{v}_n(0))\|_{S(\Hmeio)}
\end{equation*}
\begin{equation*}
=\|e^{it\Delta}K_n^M\|_{S(\Hmeio)}+\|e^{it\Delta}\widetilde{K}_n^M\|_{S(\Hmeio)} \rightarrow 0, \quad \textrm{as} \quad M,n\rightarrow \infty,
\end{equation*}
where in the last line we have used the limit \eqref{Kmn}. Therefore, applying Proposition \ref{LTstabteo} with $e=0$, we obtain $\|u_n\|_{S(\Hmeio)}+\|v_n\|_{S(\Hmeio)}<\infty$, for $n\in \mathbb{N}$ large enough, which is a contradiction with \eqref{SCunvn}.

It remains to prove that
\begin{equation}\label{MEucvc}
\quad M(u_c,v_c)=1 \quad \textrm{ and } \quad E(u_c,v_c)=M_c.
\end{equation}
By \eqref{III11} and the definition of $M_c$ we have $M(u_c,v_c)E(u_c,v_c)=M_c$, otherwise we cannot have \eqref{claim}. Thus, inequalities \eqref{III1} and \eqref{III2} imply \eqref{MEucvc}, which completes the proof.
\end{proof}

Next, we show that the flow associated to the critical solution $(u_c,v_c)$ given by Theorem \ref{ECE} enjoys some compactness properties up to a continuous shift $x(t)$ in space. More specifically we have the following proposition.

\begin{proposition}[Precompactness of the critical flow]\label{Precomp}
Let $(u_c,v_c)$ be the critical solution constructed in Theorem \ref{ECE}, then there exists a continuous path $x\in C([0,\infty);\R^3)$ such that the set
$$
\mathcal{B}:=\{(u_c(\cdot-x(t), t),v_c(\cdot-x(t), t) ): t\geq 0\}\subset H^1\times H^1
$$
is precompact in $H^1\times H^1$.
\end{proposition}
\begin{proof}
The proof is similar to that of Proposition 3.2 in \cite{dhr}. So, we only give the main steps. In $H^1\times H^1$ we let $G\simeq\R^3$ act as a translation group, that is,
$$
x_0\cdot(f_1,f_2)=(f_1(\cdot-x_0), f_2(\cdot-x_0)).
$$
Thus, in the quotient space $G\setminus H^1\times H^1$ we can introduce the metric
$$
d([f_1,f_2],[g_1,g_2]):=\inf_{x_0\in\R^3}\|(f_1(\cdot-x_0), f_2(\cdot-x_0))-(g_1,g_2)\|_{H^1\times H^1},
$$
in such a way that the proof of the proposition is equivalent to establish that the set
$$
\mathcal{C}:=\pi(\{(u_c(t),v_c(t)), \;t\geq0\})
$$
is precompact in $G\setminus H^1\times H^1$, where $\pi:H^1\times H^1\to G\setminus H^1\times H^1$ is the standard projection.

Now assume, by contradiction, that $\mathcal{C}$ is not precompact in $G\setminus H^1\times H^1$. Then, we can find a sequence $([u_c(t_n),v_c(t_n)])_{n\in\N}$ and $\varepsilon>0$ such that
\begin{equation}\label{contrGcom}
\inf_{x_0\in\R^3}\|(u_c(\cdot-x_0,t_n), v_c(\cdot-x_0,t_n))-(u_c(\cdot,t_m), v_c(\cdot,t_m))\|_{H^1\times H^1}>\varepsilon,
\end{equation}
for all $m,n\in\N$ with $m\neq n$. To obtain a contradiction with \eqref{contrGcom}, it suffices to prove that a subsequence of $(u_c(t_n), v_c(t_n))$ converges in $H^1\times H^1$.

Since $(u_c(t_n), v_c(t_n))$ is bounded in $H^1\times H^1$, we apply Theorem \ref{profdec} and Lemma \ref{Uniq} to write
\begin{equation}\label{ucdecom}
(u_c(t_n),v_c(t_n))=(e^{-it_n^1\Delta}\psi^1(\cdot-x_n^1),e^{-it_n^1\Delta}\wtpsi^1(\cdot-x_n^1))+(W_n^M,\wtw_n^M).
\end{equation}
According to the proof of Lemma \ref{Uniq} and \eqref{EnergyPsi}, we have
$$
M(\psi^1,\wtpsi^1)=1,\quad \lim_{n\to\infty}E(e^{-it_n^1\Delta}\psi^1(\cdot-x_n^1),e^{-it_n^1\Delta}\wtpsi^1(\cdot-x_n^1))=M_c.
$$
Also, in view of \eqref{pr3}, \eqref{pr4}, and the energy decomposition in \eqref{ED}, we obtain
$$
\lim_{n\to\infty}M(W_n^M,\wtw_n^M)=0,
$$
and
$$
\lim_{n\to\infty}E(W_n^M,\wtw_n^M)=0.
$$
Since $(W_n^M,\wtw_n^M)$ also satisfies \eqref{EM1}, an application of Lemma \ref{gnlemma} reveals that $E(W_n^M,\wtw_n^M)\sim A(W_n^M,\wtw_n^M)$. Hence,
$$
\lim_{n\to\infty}\|(W_n^M,\wtw_n^M)\|_{H^1\times H^1}=0.
$$
It remains to show that the first term on the right-hand side of \eqref{ucdecom} converges. To this end, it suffices to show that $\{t_n^1\}_{n\in\N}$ converges, up to a subsequence. If such a sequence does not converge, we can apply the small data global theory to obtain a contradiction as in \cite[Proposition 5.5]{hr} and \cite[Proposition 6.2]{Xu12}.
\end{proof}

As a consequence of the previous proposition we obtain the following uniform localization of the flow.

\begin{corollary}[Uniform localization]\label{UnifLoc}
Let $(u,v)$ be a solution of \eqref{3nls} and $x\in C([0,\infty);\R^3)$ a continuous path such that
$$
\mathcal{B}=\{(u(\cdot-x(t), t),v(\cdot-x(t), t) ): t\geq 0\}
$$
is precompact in $H^1\times H^1$. Then for every $\varepsilon>0$ there exists $R>0$ so that
$$
\int_{|x+x(t)|>R}\left(|\nabla u|^2 + |\nabla v|^2 +|u|^2 + |v|^2 + |u|^{4}+2\beta|uv|^{2}+|v|^{4}\right)dx<\varepsilon,
$$
for every $t\geq 0$.
\end{corollary}
\begin{proof}
Indeed, otherwise, there exist $\varepsilon>0$ and a sequence $\{t_n\}$, such that for any $R>0$ (after a change variables)
$$
\int_{|x|>R}\left(|\nabla u|^2 + |\nabla v|^2 +|u|^2 + |v|^2 + |u|^{4}+2\beta|uv|^{2}+|v|^{4}\right)(x-x(t_n),t_n)dx\geq\varepsilon.
$$
The precompactness of $\mathcal{B}$ implies  the existence of $(\phi,\psi)\in H^1\times H^1$ such that $(u(\cdot-x(t_n), t_n),v(\cdot-x(t_n), t_n))$ converges to $(\phi,\psi)$. Passing to the limit in the last inequality, we deduce
$$
\int_{|x|>R}\left(|\nabla \phi|^2 + |\nabla \psi|^2 +|\phi|^2 + |\psi|^2 + |\phi|^{4}+2\beta|\phi\psi|^{2}+|\psi|^{4}\right)(x)dx\geq\varepsilon,
$$
for any $R>0$, which contradicts the fact that $(\phi,\psi)\in H^1\times H^1$. The proof is thus completed.
\end{proof}

\section{Rigidity and the proof of Theorem \ref{globalscatters}}\label{sec6}

In this section we prove that the critical solution $(u_c, v_c)$ construct in the previous section must vanishes, which is a contradiction with 
$$
\|u_c\|_{S(\Hmeio)}+\|v_c\|_{S(\Hmeio)}=\infty.
$$
Therefore we cannot have $M_c < M(P,Q)E(P,Q)$, completing the proof of Theorem \ref{globalscatters}.

The proof is similar to that in \cite[Sections 6.3-6.5]{Xu12}, which was based on the ideas developed in \cite{dhr} in their study of the 3D cubic focusing nonlinear Schr\"odinger equation (see also \cite{fxc}). We give the details for completeness.

We start with the following lemma which asserts that the critical solution $(u_c, v_c)$ has  zero momentum (recall \eqref{mocon}).

\begin{lemma}\label{ZeroM}
Let $(u_c, v_c)$ be the critical solution construct in Theorem \ref{ECE}, then its conserved momentum $F(u_c, v_c)$ is zero.
\end{lemma}
\begin{proof}
Since $(u_c, v_c)$ is a solution of system \eqref{3nls}, by the Galilean invariance \eqref{GaIn}, for any $\xi_0\in \R^3$, we have that 
\begin{equation*}
(u_c^{\ast}(x,t), v_c^{\ast}(x,t))\equiv (e^{i(x\cdot\xi_0-t|\xi|^2)}u_c(x-2t\xi_0, t), e^{i(x\cdot\xi_0-t|\xi|^2)}v_c(x-2t\xi_0, t))
\end{equation*} 
is also a solution of \eqref{3nls}. By the definition of $(u^{\ast}_c, v^{\ast}_c)$ and Theorem \ref{ECE} (iv), it is clear that
\begin{equation}\label{uvast}
\|u^{\ast}_c\|_{S(\Hmeio)}+\|v^{\ast}_c\|_{S(\Hmeio)}=\|u_c\|_{S(\Hmeio)}+\|v_c\|_{S(\Hmeio)}=\infty.
\end{equation}
Moreover, a simple calculation yields (recall the conserved quantities \eqref{masscon}-\eqref{econ})
$$
A(u_c^{\ast}, v_c^{\ast})=|\xi_0|^2M(u_c, v_c)+2\xi_0\cdot F(u_c, v_c)+A(u_c, v_c),
$$
and thus
$$
E(u_c^{\ast}, v_c^{\ast})=\frac12|\xi_0|^2M(u_c, v_c)+\xi_0\cdot F(u_c, v_c)+E(u_c, v_c).
$$
Assume that $F(u_c, v_c)\neq 0$ and take $\xi_0=-F(u_c, v_c)/M(u_c, v_c)$. Then,
$$
M(u_c^{\ast}, v_c^{\ast})=M(u_c, v_c), \quad E(u_c^{\ast}, v_c^{\ast})=E(u_c, v_c)-\frac12\frac{F(u_c, v_c)^2}{M(u_c, v_c)},
$$
and
$$
A(u_c^{\ast}, v_c^{\ast})=A(u_c, v_c)-\frac{F(u_c, v_c)^2}{M(u_c, v_c)}.
$$
Therefore, by Theorem \ref{ECE} (i)-(iii) we deduce
$$
M(u_c^{\ast}, v_c^{\ast})E(u_c^{\ast}, v_c^{\ast})<M_c<M(P,Q)E(P,Q)
$$
and
$$
M(u^{\ast}_{c},v^{\ast}_{c})A(u^{\ast}_{c},v^{\ast}_{c})<M(P,Q)A(P,Q),
$$
which, in view of \eqref{uvast}, is a contradiction with the definition of $M_c$.
\end{proof}

Next, we provide a control of the spatial translation parameter $x(t)$ obtained in Proposition \ref{Precomp}. Indeed, we show that $x(t)$ cannot grow faster then $t$ as this parameter goes to infinity. This is the content of the following result.

\begin{lemma}\label{LinearGrow}
Let $(u_c, v_c)$ be the (global) critical solution construct in Theorem \ref{ECE} and $x(t)$ the spatial translation parameter obtained in Proposition \ref{Precomp}. Then
\begin{equation}\label{LiGr}
\frac{|x(t)|}{t}\rightarrow 0, \quad \textrm{as} \quad t\rightarrow +\infty.
\end{equation}
\end{lemma}
\begin{proof}
Suppose, by contradiction, that there exist $\delta>0$ and a sequence $t_n\rightarrow +\infty$ such that, for all $n\in \N$,
\begin{equation}\label{xtn}
\frac{|x(t_n)|}{t_n}\geq \delta.
\end{equation}
We assume, without loss of generality, that $x(0)=0$ and define
$$
\tau_n=\inf\{t\geq 0: |x(t)|\geq|x(t_n)|\},
$$
so that $0\leq \tau_n\leq t_n$ and $|x(t)|<|x(t_n)|$, for all $0\leq t < \tau_n$. The continuity of $x(t)$ yields $|x(\tau_n)|=|x(t_n)|$ and then, by \eqref{xtn}, it is clear that 
\begin{equation}\label{xtaun}
\frac{|x(\tau_n)|}{\tau_n}\geq \delta , \quad \textrm{for all}\quad n\in \N.
\end{equation}
Moreover, since $t_n\rightarrow +\infty$ we have, by \eqref{xtn}, $x(t_n)\rightarrow +\infty$, which implies $x(\tau_n)\rightarrow +\infty$. Thus,
\begin{equation}\label{taun}
\tau_n\rightarrow +\infty, \quad \textrm{as} \quad n\rightarrow \infty.
\end{equation}

By the precompactness of the set $\mathcal{B}$ defined in Proposition \ref{Precomp} and Corollary \ref{UnifLoc}, for every $\eta>0$ there exists $R(\eta)$ such that, for any $t\geq 0$,
\begin{equation}\label{Precomp2}
\int_{|x+x(t)|>R(\eta)}\Big(|\nabla u_c|^2 + |\nabla v_c|^2 +|u_c|^2 + |v_c|^2 + |u_c|^{4}+2\beta|u_cv_c|^{2}+|v_c|^{4}\Big)dx<\eta.
\end{equation}

Now, let $\theta\in C_0^{\infty}([0,\infty))$ be a real-valued function such that $0\leq \theta \leq 1$, $\|\theta'\|_{L^{\infty}}\leq 2$, $\theta(r)=1$, if $0\leq r\leq 1$ and $\theta(r)=0$, if $ r\geq 2$, and define the truncated center of mass
\begin{equation}\label{zRn}
z_{R_n}(t)=\int x\theta\left(\frac{|x|}{R_n}\right)(|u_c(x,t)|^2+|v_c(x,t)|^2)dx,
\end{equation}
where 
\begin{equation}\label{Rn}
R_n=R(\eta) + |x(\tau_n)|.
\end{equation}

We first obtain an upper bound for $z_{R_n}(0)$ and a lower bound for $z_{R_n}(\tau_n)$. Indeed, in view of \eqref{Precomp2}, we have (recall that $x(0)=0$)
\begin{equation}\label{zrn0}
\begin{split}
|z_{R_n}(0)|\leq& \int_{|x|\leq R(\eta) }|x|\theta\left(\frac{|x|}{R_n}\right)(|u_c(x,t)|^2+|v_c(x,t)|^2)dx\\
&+\int_{|x+x(0)|> R(\eta) }|x|\theta\left(\frac{|x|}{R_n}\right)(|u_c(x,t)|^2+|v_c(x,t)|^2)dx\\
\leq&R(\eta) M(u_c, v_c) + 2R_n \eta,
\end{split}
\end{equation}
where in the last line we have also used the fact that 
\begin{equation}\label{xtheta}
|x\theta\left(|x|/R\right)|\leq 2 R,
\end{equation}
for all $R> 0$.

On the other hand, 
\begin{equation*}
\begin{split}
z_{R_n}(\tau_n)=& \int_{|x+x(\tau_n)|> R(\eta) }x\theta\left(\frac{|x|}{R_n}\right)(|u_c(x,t)|^2+|v_c(x,t)|^2)dx\\
&+\int_{|x+x(\tau_n)|\leq R(\eta) }x\theta\left(\frac{|x|}{R_n}\right)(|u_c(x,t)|^2+|v_c(x,t)|^2)dx\\
\equiv & \,\,\,I + II.
\end{split}
\end{equation*}
Again using \eqref{Precomp2} and \eqref{xtheta}, we can estimate $I$ by 
$$
|I|\leq 2R_n \eta.
$$
To estimate $II$, we first note that if $|x+x(\tau_n)|\leq R(\eta)$ then $|x|\leq |x+x(\tau_n)|+|x(\tau_n)|\leq R(\eta)+|x(\tau_n)|=R_n$, which implies $\theta\left(|x|/R_n\right)|=1$. Therefore
\begin{equation*}
\begin{split}
II=& \int_{|x+x(\tau_n)|\leq R(\eta) }x(|u_c(x,t)|^2+|v_c(x,t)|^2)dx\\
= & \int_{|x+x(\tau_n)|\leq R(\eta) }(x+x(\tau_n))(|u_c(x,t)|^2+|v_c(x,t)|^2)dx\\
&-\int_{|x+x(\tau_n)|\leq R(\eta) }x(\tau_n)(|u_c(x,t)|^2+|v_c(x,t)|^2)dx\\
=& \int_{|x+x(\tau_n)|\leq R(\eta) }(x+x(\tau_n))(|u_c(x,t)|^2+|v_c(x,t)|^2)dx\\
&-x(\tau_n)M(u_c, v_c)+x(\tau_n)\int_{|x+x(\tau_n)|> R(\eta) }(|u_c(x,t)|^2+|v_c(x,t)|^2)dx\\
\equiv &\,\,\, II_a +II_b + II_c.
\end{split}
\end{equation*}

It is clear that, $|II_a|\leq R(\eta)M(u_c, v_c)$. Moreover, by \eqref{Precomp2} and \eqref{Rn}, we have $|II_c|\leq R_n \eta$. Therefore, collecting the above estimates, we deduce
\begin{equation}
\begin{split}\label{zrnt}
|z_{R_n}(\tau_n)|&\geq |II_b|-|II_a|-|II_c|- |I|\\
& \geq |x(\tau_n)|M(u_c, v_c)- R(\eta)M(u_c, v_c) - 3R_n \eta.
\end{split}
\end{equation}

Next, we claim that there exists $c>0$ such that
\begin{equation}\label{z'rnt}
\left|\frac{d}{dt}z_{R_n}(t)\right|\leq c \eta, \quad \textrm{for all} \quad t\in [0,\tau_n].
\end{equation}

Assuming the above claim, let us conclude the proof of the lemma. Indeed, the Fundamental Theorem of Calculus, \eqref{zrn0} and \eqref{zrnt} yield
\begin{equation*}
\begin{split}
c \eta \tau_n \geq \int_{0}^{\tau_n}\left|\frac{d}{dt}z_{R_n}(t)\right| dt&\geq |z_{R_n}(\tau_n)|-|z_{R_n}(0)|
\\
&\geq |x(\tau_n)|M(u_c, v_c)- 2R(\eta)M(u_c, v_c) - 5R_n \eta.
\end{split}
\end{equation*}
Thus, using \eqref{Rn} we have proved that
$$
(M(u_c, v_c)-5\eta)\frac{|x(\tau_n)|}{\tau_n}\leq c\eta +(2M(u_c, v_c)+5\eta)\frac{R(\eta)}{\tau_n}.
$$
Therefore, choosing $\eta>0$ sufficiently small such that $\eta< M(u_c, v_c)/5$ and $c\eta/ (M(u_c, v_c)-5\eta)< \delta/2$ and using \eqref{taun}, we obtain a contradiction with \eqref{xtaun}.

To complete the proof, it remains to establish \eqref{z'rnt}. First note that taking the derivative with respect to time in \eqref{zRn} and using that $(u_c,v_c)$ is a solution of \eqref{3nls}, we have
\begin{equation*}
\begin{split}
\frac{d}{dt}z_{R_n}(t)=&\,\,\,2Im\left\{\int \theta\left(\frac{|x|}{R_n}\right)\left(\overline{u}_c(t)\nabla u_c(t)+\overline{v}_c(t)\nabla v_c(t)\right)dx\right\}\\
&+2Im\left\{\int \frac{x}{|x|R_n}\theta'\left(\frac{|x|}{R_n}\right)x\cdot\left(\overline{u}_c(t)\nabla u_c(t)+\overline{v}_c(t)\nabla v_c(t)\right)dx\right\}\\
=&\,\,\,2Im\left\{\int \left[1-\theta\left(\frac{|x|}{R_n}\right)\right]\left(\overline{u}_c(t)\nabla u_c(t)+\overline{v}_c(t)\nabla v_c(t)\right)dx\right\}\\
&+2Im\left\{\int \frac{x}{|x|R_n}\theta'\left(\frac{|x|}{R_n}\right)x\cdot\left(\overline{u}_c(t)\nabla u_c(t)+\overline{v}_c(t)\nabla v_c(t)\right)dx\right\},
\end{split}
\end{equation*}
where in the last line we have used the zero momentum property in Lemma \ref{ZeroM}. By the definition of $\theta$, it is easy to see that
$$
\left[1-\theta\left(\frac{|x|}{R_n}\right)\right]=0, \quad \textrm{for all} \quad |x|\leq R_n
$$
and (since $0\leq \theta \leq 1$)
$$
\left|1-\theta\left(\frac{|x|}{R_n}\right)\right|\leq 2, \quad \textrm{for all} \quad |x|> R_n.
$$
Moreover
$$
\theta'\left(\frac{|x|}{R_n}\right)=0, \quad \textrm{for all} \quad |x|\leq R_n \quad \textrm{and} \quad |x|\geq 2R_n.
$$
and (since $\|\theta'\|_{L^{\infty}}\leq 2$)
$$
\frac{|x|}{R_n}\left|\theta'\left(\frac{|x|}{R_n}\right)\right|\leq 4, \quad \textrm{for all} \quad R_n\leq |x|\leq 2R_n .
$$
Therefore, there exists $c>0$ such that
\begin{equation}\label{ddtz}
\begin{split}
\left|\frac{d}{dt}z_{R_n}(t)\right|&\leq \frac{c}{2} \int_{|x|>R_n}\left(|\overline{u_c}(t)\nabla u_c(t)|+|\overline{v_c}(t)\nabla v_c(t)|\right)dx\\
&\leq c \int_{|x|>R_n}\left(|\nabla u_c|^2 + |\nabla v_c|^2 +|u_c|^2 + |v_c|^2\right)dx.
\end{split}
\end{equation}
Finally, note that for $0\leq t\leq \tau_n$ and $|x|>R_n$ we have $|x+x(t)|\geq R_n - |x(\tau_n)|=R(\eta)$ by \eqref{Rn}, and thus, from \eqref{ddtz} and \eqref{Precomp2} we deduce \eqref{z'rnt}. The proof of the lemma is thus completed.
\end{proof}

Now we are in a position to prove the main result of this section, our rigidity (or Liouville-type) theorem.

\begin{theorem}[Rigidity Theorem]\label{Rigidity}
Assume that $(u_0,v_0)\in H^1\times H^1$ has zero momentum $F(u_0,v_0)=0$ and satisfies conditions \eqref{EM} and \eqref{EM1}. Let $(u,v)$ be the global solution of system \eqref{3nls} with initial data $(u_0,v_0)$ given by Theorem \ref{globalPastor} and suppose that 
$$
\mathcal{B}=\{(u(\cdot-x(t), t),v(\cdot-x(t), t) ): t\geq 0\}
$$
is precompact in $H^1\times H^1$, for some continuous path $x\in C([0,\infty);\R^3)$. Then $u_0, v_0$ are both zero.
\end{theorem}

The main tool in the proof of Theorem \ref{Rigidity} is the following local version of the Virial identity.

\begin{lemma}\label{viteo}
Let $(u,v)\in C((-T_*,T^*);H^1\times H^1)$ be a solution of
\eqref{3nls} and $\varphi\in C_0^\infty(\R^3)$. Define
$$
V(t)=\int \varphi(x)(|u|^2+|v|^2)dx.
$$
Then,
\begin{equation}\label{vlinha}
V'(t)=2Im\int\nabla \varphi\cdot\nabla u\overline{u}\,dx+ 2Im\int\nabla
\varphi\cdot\nabla v\overline{v}\,dx
\end{equation}
and
\begin{equation}\label{vi1}
\begin{split}
V''(t)=&4\sum_{k,j=1}^3Re\int\frac{\p^2\ff}{\p x_k\p
x_j}(\p_{x_k}u\p_{x_j}\ub+\p_{x_k}v\p_{x_j}\vb)dx-\int\Delta^2\ff(|u|^2+|v|^2)dx\\
&-\int \Delta
\ff(|u|^{4}+2\beta|uv|^{2}+|v|^{4})dx,
\end{split}
\end{equation}
where $\p_{x_k}$ indicates the partial derivative with respect to $x_k$.
\end{lemma}
\begin{proof}
The proof follows the same steps as in \cite[Lemma 2.9]{ka}. So we omit the details.  
\end{proof}

\begin{proof}[Proof of Theorem \ref{Rigidity}]
Let $\zeta\in C_0^{\infty}(\R^3)$ be a radial function such that $\zeta(x)=|x|^2$, if $|x|\leq 1$ and $\zeta(x)=0$, if $ |x|\geq 2$. For $R>0$ define 
\begin{equation}\label{zR}
z_{R}(t)=\int R^2\zeta\left(\frac{x}{R}\right)(|u(x,t)|^2+|v(x,t)|^2)dx.
\end{equation}

Our goal is to study $z_{R}(t)$ for large $R>0$ in a suitable time interval. We start with an upper bound for the derivative of $z_{R}(t)$. Indeed, from \eqref{vlinha}, we get
$$
\frac{d}{dt}z_{R}(t)=2 Im\left\{\int R\nabla \zeta\left(\frac{x}{R}\right)\cdot\nabla u(t)\overline{u}(t) + R\nabla
\zeta\left(\frac{x}{R}\right)\cdot\nabla v(t)\overline{v}(t)dx\right\}.
$$
Therefore, since $\nabla\zeta\left(x/R\right) = 0$ if $ |x|\geq 2R$, we can apply the H\"older's inequality and \eqref{EM2} to obtain
\begin{equation}\label{dtzR}
\begin{split}
\left|\frac{d}{dt}z_{R}(t)\right|\leq& \,\, cR\int_{|x|\leq 2R}(|\nabla u(t)||{u}(t)|+|\nabla v(t)||{v}(t)|)dx\\
\leq & \,\, cR(\|\nabla u(t)\|_{L^2}\|u(t)\|_{L^2}+\|\nabla v(t)\|_{L^2}\|v(t)\|_{L^2})\\
\leq & \,\, 2cR(M(u,v)A(u(t), v(t))^{1/2}\leq AR,
\end{split}
\end{equation}
where $A=2c(M(P,Q)A(P, Q))^{1/2}$ and thus it is independent of $R>0$ and $t\geq 0$.

Next, we establish a lower bound for the derivative of $z_{R}(t)$. To this end, we use \eqref{vi1} to conclude
\begin{equation*}
\begin{split}
\frac{d^2}{dt^2}z_{R}(t)=&\,\, 4\sum_{k,j=1}^3Re\int\frac{\p^2\zeta}{\p x_k\p
x_j}\left(\frac{x}{R}\right)(\p_{x_k}u(t)\p_{x_j}\ub(t)+\p_{x_k}v(t)\p_{x_j}\vb(t))dx\\
&-\frac{1}{R^2}\int\Delta^2\zeta\left(\frac{x}{R}\right)(|u(t)|^2+|v(t)|^2)dx\\
&-\int \Delta
\zeta\left(\frac{x}{R}\right)(|u(t)|^{4}+2\beta|u(t)v(t)|^{2}+|v(t)|^{4})dx.
\end{split}
\end{equation*}

Now, using that $\zeta$ is radial and compactly supported, we can rewrite the above expression as (recall  \eqref{Ruv} and \eqref{acon})
\begin{equation}\label{d2zr}
\frac{d^2}{dt^2}z_{R}(t)=S(u,v)(t)+ Z_R(u,v)(t),
\end{equation}
where $S$ is defined in \eqref{Ruv} and
\begin{equation*}
\begin{split}
Z_R(u,v)(t)=&\,\, 4\sum_{j=1}^3Re\int\left[\frac{\p^2\zeta}{\p x^2_j}\left(\frac{x}{R}\right)-2\right](|\p_{x_j}u(t)|^2+|\p_{x_j}v(t)|^2)dx\\
&+4\sum_{j\neq k}^3Re\int_{R\leq |x|\leq 2R}\frac{\p^2\zeta}{\p x_k\p
x_j}\left(\frac{x}{R}\right)(\p_{x_k}u(t)\p_{x_j}\ub(t)+\p_{x_k}v(t)\p_{x_j}\vb(t))dx\\
&-\frac{1}{R^2}\int\Delta^2\zeta\left(\frac{x}{R}\right)(|u(t)|^2+|v(t)|^2)dx\\
&-\int \left[\Delta
\zeta\left(\frac{x}{R}\right)-6\right](|u(t)|^{4}+2\beta|u(t)v(t)|^{2}+|v(t)|^{4})dx.
\end{split}
\end{equation*}

Since $\frac{\p^2\zeta}{\p x^2_j}\left(\frac{x}{R}\right)=2$ and $\Delta^2\zeta\left(\frac{x}{R}\right)=0$, if $|x|< R$, we obtain the following bound for $Z_R(u,v)$.
\begin{equation}\label{ZR}
|Z_R(u,v)(t)|\leq B\int_{|x|\geq R}\left(|\nabla u|^2 + |\nabla v|^2 +\frac{1}{R^2}[|u|^2 + |v|^2] + |u|^{4}+2\beta|uv|^{2}+|v|^{4}\right)dx,
\end{equation}
where the constant $B>0$ is independent of $R>0$ and $t\geq 0$.

On the other hand, we have a lower bound for $S(u,v)(t)$ from Lemma \eqref{grolemma} (iii). In fact,
\begin{equation}\label{Suvt}
 S(u,v)(t)\geq \omega A(u(t),v(t))\geq 2\omega E(u,v) ,
\end{equation}
where $\omega=8\left[1-\left(\frac{M(u,v)E(u,v)}{M(P,Q)E(P,Q)}\right)^{1/2}\right]$ (note that $\omega>0$ since $(u,v)$ satisfies \eqref{EM}).
Collecting \eqref{d2zr}-\eqref{Suvt}, we deduce
\begin{equation}\label{d2zr2}
\frac{d^2}{dt^2}z_{R}(t)\geq 2\omega E(u,v)-|Z_R(u,v)(t)|.
\end{equation}
Moreover, Corollary \ref{UnifLoc} with $\varepsilon(\omega)=\omega E(u,v)/B$ implies the existence of $R(\omega)$ such that
\begin{equation}\label{Precomp3}
\int_{|x+x(t)|>R(\omega)}\left(|\nabla u|^2 + |\nabla v|^2 +|u|^2 + |v|^2 + |u|^{4}+2\beta|uv|^{2}+|v|^{4}\right)dx<\frac{\omega E(u,v)}{B},
\end{equation}
for every $t\geq 0$.

Next, from Lemma \ref{LinearGrow} there exists $t_0>0$ such that 
$$
|x(t)|\leq \frac{\omega E(u,v)}{4A}t,
$$
for all $t\geq t_0$, where $A$ is the constant in \eqref{dtzR}. Let $t_1>t_0$ and define
\begin{equation}\label{Rt1}
R_{t_1}=R(\omega)+\frac{\omega E(u,v)}{4A}t_1.
\end{equation}
Choosing $t_1>0$ sufficiently large we can assume $R_{t_1}>1$. Moreover, if $t\in [t_0,t_1]$ and $|x|\geq R_{t_1}$ then $|x+x(t)|\geq |x|-|x(t)|\geq R_{t_1}-\frac{\omega E(u,v)}{4A}t_1=R(\omega)$. Therefore by \eqref{ZR} and \eqref{Precomp3} we conclude, for all $t\in [t_0,t_1]$, that
$$
|Z_{R_{t_1}}(u,v)(t)|\leq \omega E(u,v)
$$
and, from \eqref{d2zr2},
\begin{equation}\label{d2zr3}
\frac{d^2}{dt^2}z_{R_{t_1}}(t)\geq \omega E(u,v).
\end{equation}
Applying the Fundamental Theorem of Calculus, from \eqref{dtzR}, \eqref{d2zr3} and \eqref{Rt1}, we finally get
\begin{equation*}
\begin{split}
\omega E(u,v)(t_1-t_0)\leq& \int_{t_0}^{t_1}\frac{d^2}{dt^2}z_{R_{t_1}}(t)dt\leq \,\, \left|\frac{d}{dt}z_{R_{t_1}}(t_1)-\frac{d}{dt}z_{R_{t_1}}(t_0)\right|\\
\leq& \,\, 2AR_{t_1}=2AR(\omega)+\frac{\omega E(u,v)}{2}t_1.
\end{split}
\end{equation*}
Letting $t_1\rightarrow \infty$ in the last inequality we obtain a contradiction unless $E(u,v)=0$, which implies, by Lemma \ref{grolemma} (i), that $u_0, v_0$ are both zero.
\end{proof}

\section{Appendix}

In this short Appendix we will prove that the relation \eqref{subxu} holds.

\begin{proposition}\label{AP}
We have
$$
K^+\varsubsetneq \mathcal{K}.
$$
\end{proposition}
\begin{proof}
The proof that $K^+\subset \mathcal{K}$ is similar to that in the Appendix of \cite{Xu12}, where the author gave the proof in the case of the 3D cubic Schr\"odinger equation, comparing his result with the one in \cite{dhr}. Indeed, from the results in \cite[Section 3]{fm} (see also \cite[Remark 2.5]{P2015}), we have
\begin{equation}\label{appendix1}
A(P,Q)=3M(P,Q), \qquad \|P\|_{L^4}^4+2\beta\|PQ\|_{L^2}^2+\|Q\|_{L^4}^4=4M(P,Q).
\end{equation}
These two equalities combine to give
\begin{equation}\label{appendix2}
E(P,Q)=\dfrac{1}{2}M(P,Q)
\end{equation}
and
\begin{equation}\label{appendix3}
J(P,Q)=E(P,Q)+\dfrac{1}{2}M(P,Q)=M(P,Q).
\end{equation}
Now assume that $(u_0,v_0)\in K^+\setminus\{(0,0)\}$. Then,
$$
E(u_0,v_0)=\dfrac{1}{6}A(u_0,v_0)+\dfrac{1}{6}K(u_0,v_0)\geq0.
$$
By using that geometric mean is no larger than arithmetic mean, that is, $2\sqrt{ab}\leq a+b$, for any $a,b\geq0$, we obtain
$$
2\sqrt{\dfrac{1}{2}M(u_0,v_0) E(u_0,v_0)}\leq \dfrac{1}{2}M(u_0,v_0)+ E(u_0,v_0)=J(u_0,v_0)<J(P,Q)=M(P,Q),
$$
where we have used \eqref{appendix3} in the last equality.
Hence, in view of \eqref{appendix2},
\begin{equation}\label{appendix4}
M(u_0,v_0) E(u_0,v_0)<\dfrac{1}{2}M(P,Q)^2=M(P,Q)E(P,Q).
\end{equation}
On the other hand,
\[
\begin{split}
2\sqrt{\dfrac{1}{6}A(u_0,v_0)\cdot \dfrac{1}{2} M(u_0,v_0)}& \leq \dfrac{1}{6}A(u_0,v_0)+\dfrac{1}{2} M(u_0,v_0)\\
&\leq \dfrac{1}{6}A(u_0,v_0)+\dfrac{1}{2} M(u_0,v_0)+\dfrac{1}{6}K(u_0,v_0)\\
&=E(u_0,v_0)+\dfrac{1}{2} M(u_0,v_0)\\
&=J(u_0,v_0)<J(P,Q)=M(P,Q).
\end{split}
\]
Therefore, in view of \eqref{appendix1},
\begin{equation}\label{appendix5}
M(u_0,v_0) A(u_0,v_0)<3M(P,Q)^2=M(P,Q)A(P,Q).
\end{equation}
Identities \eqref{appendix4} and \eqref{appendix5} then show that $(u_0,v_0)\in \mathcal{K}$.

To see that $K^+\neq\mathcal{K}$, the scaling invariance of the quantities $M(u,v)E(u,v)$ and $M(u,v)A(u,v)$ play a crucial role. Indeed, fix any $(u_0,v_0)\in \mathcal{K}\setminus\{(0,0)\}$. Then, as we already observed in the introduction, it is clear that, for any $\lambda>0$,
$$
(u_{0\lambda},v_{0\lambda})\equiv(\lambda u_0(\lambda \cdot),\lambda v_0(\lambda \cdot))\in \mathcal{K}.
$$
On the other hand,
$$
J(u_{0\lambda},v_{0\lambda})=E(u_{0\lambda},v_{0\lambda})+\dfrac{1}{2}M(u_{0\lambda},v_{0\lambda})=\lambda E(u_0,v_0)+\dfrac{1}{2\lambda}M(u_0,v_0).
$$
From Lemma \ref{grolemma} (i), we obtain $E(u_{0},v_{0})>0$. As a consequence, $J(u_{0\lambda},v_{0\lambda})\to\infty$, as $\lambda\to \infty$, which means that, for $\lambda$ sufficiently large, $(u_{0\lambda},v_{0\lambda})$ cannot satisfy $J(u_{0\lambda},v_{0\lambda})<J_0$ and then it does not belong to $K^+$. Thus Proposition \ref{AP} is proved.
\end{proof}

\section*{Acknowledgement}

L.G. Farah was partially supported by CNPq/Brazil and FAPEMIG/Brazil. A. Pastor was partially supported by CNPq/Brazil and FAPESP/Brazil.

\bibliographystyle{mrl}

\end{document}